\documentclass{article}
\usepackage{a4wide}
\usepackage{amsfonts}
\usepackage[margin=2cm]{geometry}
\usepackage{graphics, amsmath,enumitem, comment, color}
\usepackage{graphicx}
\usepackage{gnuplottex}
\usepackage{epsfig}
\usepackage{amssymb}
\usepackage{amsmath,amsthm}
\usepackage{mathrsfs}
\usepackage{colortbl}
\usepackage{tikz}
\usetikzlibrary{positioning,arrows}
\usetikzlibrary{3d}
\usetikzlibrary[patterns]
\usepackage{pgflibraryarrows}
\usepackage{pgflibrarysnakes}
\usetikzlibrary{snakes}
\usepackage{arydshln}
\usepackage{graphicx}
\usepackage{setspace}
\usepackage{eucal}

\newif\ifcolorcomments
\newcommand{\allowcomments}[4]{
\newcommand{#1}[1]{\ifdraft{\ifcolorcomments{\textcolor{#4}{##1 --#3}}\else{\textsl{ ##1 \ --#3}}\fi}\else{}\fi}
}

\allowcomments{\comAB}{Ayreena}{A}{magenta}
\allowcomments{\comCC}{Carlo}{C}{orange}
\allowcomments{\comDS}{Stefano}{S}{green}
\colorcommentstrue
\def\bc{\begin{center}}
\def\ec{\end{center}}
\def\be{\begin{equation}}
\def\ee{\end{equation}}

\def\N{\mathbb N}
\def\Z{\mathbb Z}
\def\Q{\mathbb Q}

\def\R{\mathbb R}

\def\ah{\hat{\alpha}}

\def\Z{\mathbb Z}

\newtheorem{lem}{Lemma}[section]

\newtheorem{dfn}[lem]{Definition}
\newtheorem{pro}[lem]{Proposition}
\newtheorem{thm}[lem]{Theorem}

\newtheorem{con}[lem]{Conjecture}
\newtheorem{exa}[lem]{Example}
\newtheorem{cor}[lem]{Corollary}
\newtheorem{rem}[lem]{Remark}
\numberwithin{equation}{section}
\newtheorem{que}[lem]{Question}
\newtheorem{step}{Step}

\newif\ifdraft\drafttrue

\allowdisplaybreaks

\date{}

\title{Global and local minima of $\alpha$-Brjuno functions}

\author{A. Bakhtawar
\and
C. Carminati
\and
S. Marmi
}

\newcommand{\Addresses}{{
  \bigskip
 \footnotesize

  A.~Bakhtawar, \textsc{
  Centro di Ricerca Matematica Ennio De Giorgi, 
 Scuola Normale Superiore, Piazza dei Cavalieri 3,
56126 Pisa, Italy               
  }\par\nopagebreak
  \textit{E-mail address}, A.~Bakhtawar: \texttt{ayreena.bakhtawar@sns.it}

  \medskip

  C.~Carminati, \textsc{Dipartimento di Matematica, University di Pisa, Largo Bruno Pontecorvo 5, 56127
Pisa, Italy}\par\nopagebreak
  \textit{E-mail address}, C.~Carminati: \texttt{carlo.carminati@unipi.it}

  \medskip

  S.~Marmi, \textsc{Scuola Normale Superiore, Piazza dei Cavalieri 7, 56126 Pisa, Italy}\par\nopagebreak
  \textit{E-mail address}, S.~Marmi: \texttt{stefano.marmi@sns.it}

}}

\begin{document}

\maketitle

\begin{abstract}
The main goal of this article is to analyze some peculiar features of the global (and local)  minima of $\alpha$-Brjuno functions $B_\alpha$ where $\alpha\in(0,1].$ Our starting point is the result by Balazard--Martin (2020), who showed that the minimum of $B_1$ is attained at $g:=\frac{\sqrt 5 -1}{2}$; analyzing the scaling properties of $B_1$ near $g$ we shall deduce that all preimages of $g$ under the Gauss map are also local minima for $B_1$. Next we consider the problem of characterizing global and local minima of $B_\alpha$ for other values of $\alpha$: we show that for $\alpha\in (g,1)$ the global minimum is again attained at $g$, while for $\alpha$ in a neighbourhood of $1/2$ the function $B_{\alpha}$ attains its minimum at $\gamma:=\sqrt{2}-1$. 
The fact that the minimum of $B_\alpha$ is attained when $\alpha$ ranges a whole interval of parameters is non trivial. Indeed, we prove that  $B_{\alpha}$ is lower semicontinuous for all rational $\alpha,$ but  we also exhibit an irrational $\alpha$ for which $B_{\alpha}$ is not lower semicontinuous.
\end{abstract}

\section{Introduction}

Let $x\in R\setminus \Q$ and let $\left\{    \frac{p_{n}}{ q_{n} }    \right\}_{n\geq 0}$ be the sequence of the convergents of its continued fraction expansion. A Brjuno number is an irrational number $x$ such that $\sum^{\infty}_{n=0}\frac{\log q_{n+1}}{q_{n}}<\infty.$ Almost all
real numbers are Brjuno numbers, since for all Diophantine numbers one has $q_{n+1}=\mathbb{O}(q^{\tau+1}_{n})$ for some $\tau\geq 0.$ But some Liouville numbers also verify the Brjuno condition, e.g. $\sum^{\infty}_{n=0}10^{-n!}.$    The importance of Brjuno numbers comes from the study of one-dimensional
analytic small divisor problems in dimension one. 
In the case of germs of holomorphic diffeomorphisms
of one complex variable with an indifferent fixed point, extending
a previous result of Siegel \cite{Si_42}, Brjuno proved \cite{Br_71} that all germs with linear
part $\lambda=e^{2\pi i x}$ are linearizable if $x$ is a Brjuno number. The most famous results are due to Yoccoz \cite{Yo_95}, who proved
that the Brjuno condition is also necessary: if $x$  is not a Brjuno number then there exist non-linearizable analytic germs with linear part $\lambda=e^{2\pi i x}$ (indeed  one can just take the quadratic polynomial). 
Similar results hold for the local conjugacy problem of analytic diffeomorphisms of the circle
\cite{Yo_02} and for 
 some complex area–preserving maps \cite{ChMa_22, Da_94, Mar_90}. In a somewhat
different but closely related context, it is conjectured that the Brjuno conditon is optimal for
the existence of real analytic invariant circles in the standard family \cite{Mac_88,MaJa_92}.
Yoccoz’s work used a rigorous renormalization technique and the $PGL_{2}(\Z)$ action on $\R\setminus\Q$
plays an important role: he thus found convenient to reformulate the Brjuno condition in
terms of an arithmetical function, the Brjuno function, which satisfies a cocycle equation
\cite{MaMoYo_01} under this action, see \cite[Appendix 5]{MaMoYo_01}.
The set of Brjuno numbers  can be characterized as the set where the Brjuno
function $B:R\setminus \Q \to \R \cup \{ +\infty\}$ is finite, thus it is 
invariant under the action of the modular
group $PGL (2, \Z)$.  
Another important consequence of Yoccoz's work is that the Brjuno
function gives the size (modulus $L^{\infty}$ \cite{Yo_95}, and even continuous \cite{BC}, functions) of the domain of stability
around an indifferent fixed point of a quadratic polynomial. It conjecturally plays the same
role in many other small divisor problems \cite{MA_00,MaJa_92,  MaYa_02, MaCa08}.

\subsection{$\alpha$-Brjuno maps: old and new results}
Let $\alpha\in[0,1],$ $\bar{\alpha}=\max(\alpha,1-\alpha)$ so that $\frac{1}{2}\leq {\bar\alpha}\leq1$; let $$I_\alpha:=\begin{cases}
    [0,\alpha) & \alpha>1/2\\
    [0, \bar{\alpha}] & \alpha \leq 1/2.
\end{cases}$$

Let us consider  the one-parameter family of maps\footnote{$A_{\alpha}$ is just the folded version of $\tilde{A_{\alpha}}:[1-\alpha,\alpha)\to [1-\alpha,\alpha)$ (see\cite{NaNa_08}), namely $| \tilde{A_{\alpha}}(x)|=A_{\alpha}(|x|)$ where 
 $$\tilde{A_{\alpha}}=\left\vert\frac{1}{x}\right\vert-c_\alpha (x), \ \ \ \mbox{with } c_\alpha(x)=\left[  \left\vert \frac{1}{x}\right\vert   +1-\alpha\right] \ \ (c_\alpha \mbox{ is the unique integer choice so that} \tilde{A_{\alpha}}(x) \in [\alpha-1,\alpha[) .$$ }   $A_{\alpha}:I_\alpha\to I_\alpha$ defined by $A_{\alpha}(0)=0$ and 
 \begin{equation*}
 A_{\alpha}(x)=\left\vert\frac{1}{x}-\left[  \frac{1}{x} -\alpha+1 \right]\right\vert, \; \; \mbox{for } x\neq 0,
 \end{equation*}
  where  $[x]$ denotes the integer part of $x.$ The maps in the family $(A_\alpha)_{\alpha \in [0,1]}$ are a quite natural generalization of the Gauss map, which in fact appears in the family for $\alpha=1$. By the way, when dealing with a value $\alpha$ in parameter space, we shall often code it using the regular continued fraction expansion, which is the symbolic orbit of $\alpha$ under the Gauss map $A_1$.

Following \cite{MaMoYo_97},  one can introduce the  (generalized) Brjuno function (see \cite{LuMar_10,MaMoYo_97,MaMoYo_06,MoCaMa_99}, but also \cite{Yo_95} where the case $\alpha=1/2$  was originally introduced) as follows. 
If $x\in \R$ let us define $x_0\in I_\alpha$ as $x_0=|x-[x]_{\bar{\alpha}}|,$ where $[x]_{\bar{\alpha}}=[x+1-\bar{\alpha}]$; the point  $x_0$ will be called the {\em canonical representative} of $x$ in $I_\alpha$. Then we set
\begin{equation*}
    x_n:=A^n_\alpha(x_0), \ \ \ \beta_{-1}:=1, \ \ \beta_n:=x_0x_1...x_n.
\end{equation*}
Thus we can define 
\begin{equation}\label{aBFun}
B_{\alpha}(x)=\begin{cases}
    +\infty & x\in \mathbb{Q}\\
    \sum^{\infty}_{j=0} \beta_{j-1}(x)\log (1/x_{j}) & \mbox{ otherwise.}
\end{cases}
\end{equation}
By definition $B_\alpha$ is a 1-periodic function defined on the real line, and it satisfies the functional equation
\begin{equation}\label{BFE}
B_{\alpha}(x)=-\log(x)+xB_{\alpha}(A_{\alpha}(x)) \text{ for  all } x\in I_\alpha
\end{equation}
and more generally,
\begin{equation} \label{gen}
B_{\alpha}(x)=B_{\alpha}^{(K)}(x)+\beta_{K}(x)B_{\alpha}(A_{\alpha}^{K+1}(x)) \quad (K\in \N, x\in I_\alpha),
\end{equation}
where $B_{K}$ denotes the partial sum associated  to the $\alpha$-continued fraction expansion
\begin{equation}\label{PS}
B_{\alpha}^{(K)}(x)=\sum_{j=0}^{K}\beta_{j-1}{(x)}\log(1/A_{\alpha}^{j}(x)).
\end{equation}

The local properties of $B_{1}$ were studied in \cite{BaMa_12}, where the authors showed that the Lebesgue points of the Brjuno function $B_{1}$ are exactly the Brjuno numbers. Additionally, the multifractal analysis of $B_{1}$ was carried out in \cite{StMa_18}.
More recently, Balazard and Martin \cite{BaMa_20} proved that for $\alpha=1,$ the Brjuno function is lower semicontinuous and attains a global minimum.
 \begin{thm}[\cite{BaMa_20}] \label{BM} 
 Let $g=\frac{\sqrt 5 -1}{2}$ denote the golden number. Then
 $$\min_{x\in [0,1]}B_{1}(x)=B_{1}(g).$$
 \end{thm}

\begin{figure}
 \begin{center}
\includegraphics[width=0.8\textwidth]{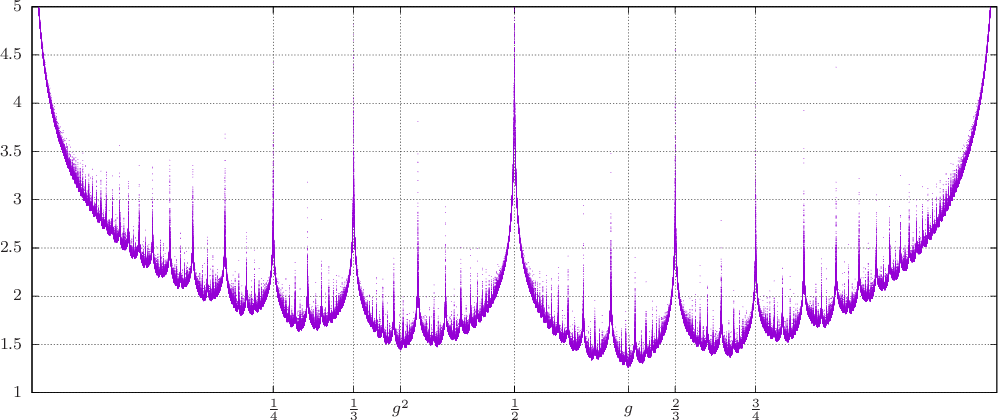}
 \caption{The graph of the  Brjuno function $B_{1}$ associated to Gauss map.}
\label{B1}
\end{center}
\end{figure}

 %

It is natural to also  try to characterize the other local minima of $B_1.$ For instance,  numerical evidence suggests that $\min_{x\in [0,1/2]}B_{1}(x)=B_{1}(g^2),$ and indeed this can be proven with an elementary computation relying on Theorem \ref{BM} (see figure \ref{B1} 
, and Corollary \ref{CBM} in Section \ref{Ca1} for a proof).

Note that both $g$ and $g^{2}$ are ``noble'' numbers in the sense of the following definition.
 \begin{dfn}\label{noble}
 Let $A_1$ be the Gauss map, the set of noble numbers $\mathcal{N}$ consists of all inverse images of $g$ under $A_1$:
 $$ \mathcal{N}=\cup_k A_1^{-k}(g).$$
 \end{dfn}

On the basis of numerical evidence, one is naturally led to believe in the following conjecture.
   \begin{con}
Let $\mathcal{M}$ denote the set of local minima of $B_{1}$. Then $$\mathcal{M}=\mathcal{N}.$$ 
  \end{con} 

We will provide a partial answer to this conjecture. The first step in this direction 
is to prove  that the Brjuno function $B_{1}$ has a cusp-like minimum at $g,$ namely
 \begin{thm}\label{AIM}
There exists $c>0$ such that
\begin{equation*}
B_{1}(x)-B_{1}(g)\geq c|x-g|^{1/2} \text{ for all } x\in(0,1).
\end{equation*}
\end{thm}

It is then not too hard to imagine that one can use the functional equation \eqref{gen} to propagate the cusp-like property of the global minimum to all other noble numbers, gaining the following result:
\begin{cor}
 Let $\nu$ be a noble number.  Then $\nu$ is a local minimum of $B_{1}.$
 \end{cor}
This proves the inclusion $\mathcal{N}\subset \mathcal{M},$ which establishes one half of the conjecture. Unfortunately, we still miss an argument to get the other inclusion.

It is natural to ask what can be said, in general, about the problem of determining the minimum of $B_\alpha$ for general $\alpha<1.$ Note that, in this case, even the existence of a global minimum is a nontrivial issue, since there exist (irrational) values of $\alpha$ for which the semicontinuity property fails (see section \ref{lscfail}). And even for rational values of $\alpha,$ when the minimum of $B_\alpha$ exists (see \ref{LSCa}), the location (and the value) of the minimum varies in a strange manner. If we set $\min_x B_\alpha(x)= B_\alpha(r_\alpha),$ numerical evidence suggests that the functions $\alpha \mapsto r_\alpha$ and $\alpha \mapsto B_\alpha(r_\alpha)$ are almost everywhere constant (see figure \ref{figureminima}). Moreover the structure of the connected components of the set on which these functions are constant seems very complicated, and a general result  still seems out of reach.
\begin{figure}[h]
 \begin{center}
 \includegraphics[width=0.8\textwidth]{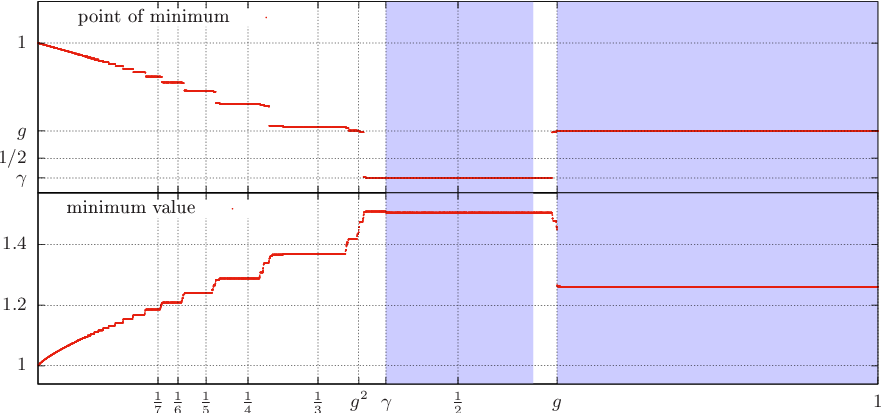}
 \caption{(Approximate) behaviour of the minimum point and minimum value as $\alpha$ ranges in $[0,1]$; shaded regions correspond to parameter values for which the numerical evidence is confirmed by our analytical results.}
\label{figureminima}
\end{center}
\end{figure}

However, we can provide some partial results that confirm some of the features observed numerically.
In the first place, we show that we can directly apply Theorem \ref{AIM} to prove that the minimum remains unchanged when $\alpha$ ranges in a left neighbourhood of $1.$ This leads to the following generalization of Theorem  \ref{BM}.
\begin{thm}\label{Bingo}
For all $\alpha\in(g,1],$ we have that $B_{\alpha}(g)=B_{1}(g)$ and $B_\alpha(x)\geq B_1(x)$. Therefore,
\begin{equation*}
\min_{x\in[0,\alpha]}B_{\alpha}(x)=B_{\alpha}(g).
\end{equation*}
\end{thm}

For these values of the parameter, $g$ is a cusp-like minimum  for $B_\alpha,$ and this property propagates by means of the functional equation (see Corollary  \ref{immediate}).


For parameter values $\alpha \le g$ we can prove the following result:  
\begin{thm}\label{monotonic}
Let\footnote{$\gamma$ is sometimes called the ``silver number"}
    $\gamma:=\sqrt{2}-1=[0;\overline{2}]$. There exists   $\gamma^*\approx 0.5895...>1-\gamma$ such that, for any  fixed $x$, the map $\alpha \mapsto B_\alpha(x)$ is monotone decreasing for $\alpha \in (\gamma, 1/2]$, monotone increasing for $\alpha \in [1/2, \gamma^*]$, and again monotone decreasing for $\alpha \in [\gamma^*,g]$. 
    Moreover, for almost every $x$ the monotonicity is strict.
\end{thm}

For $\alpha=1/2,$ we  can repeat the argument of Balazard and Martin \cite{BaMa_12} to show  that $\min_{x\in[0,1/2]} B_{1/2}(x)=B_{1/2}(\gamma)$ where $\gamma$ is as defined above; in this case also the scaling argument works in the same way as for $\alpha=1$, and this implies that $\gamma$ is a cusp-like minimum and all the preimages of $\gamma$ under $A_{1/2}$ are local minima (see section \ref{Ca2}). 

Finally, due to the Theorem \ref{monotonic}, the same argument as in the proof of Theorem \ref{Bingo} implies that $\min_{x} B_{\alpha}(x)=B_{\alpha}(\gamma)$ for all $\alpha \in (\gamma, \gamma^*)$.

The structure of the paper is the following: in Section \ref{notation}, we introduce the necessary notations and we prove some preliminary results. Section \ref{SLSC} is devoted to the study of the lower semicontinuity of the $\alpha$-Brjuno functions and some other variants of  Brjuno functions. In Section \ref{Ca1}, we present the prove of  Theorem \ref{AIM} while Theorem \ref{Bingo} is proved in Section 5. We conclude our paper with a brief discussion of numerical results and some conjectures.

\section{Notations and preliminary results }
\label{notation}

Given any $\alpha\in[0,1],$ each $x\in \R$ has a canonical representative  $x_{0} \in I_\alpha,$ defined as  $x_0=|x-[x]_{\bar{\alpha}}|.$ Let us also define the quantities
$a_{\alpha,0}=[x]_{\bar{\alpha}},$     $\epsilon_{\alpha,0}(x)=sign{(x-[x]_{\bar{\alpha}})}$. Every
$x_0\in I_{\alpha}\setminus \mathbb{Q}$ has an infinite $\alpha$-continued fraction given by the symbolic orbit of $x_0$ under the iteration of the transformation $A_{\alpha}$ as follows: for $n\geq 0,$ let us set
 $$ x_{n}=A^n_{\alpha}(x_{0}), \quad
  a_{n+1}=\left[\frac{1}{x_{n}} -\alpha+1\right], \quad \epsilon_{n+1}:= {\rm sign} (x^{-1}_n -a_{n+1}).$$
 Then
  $x_{n}^{-1}=a_{n+1}+\epsilon_{n+1}x_{n+1}  $  and 
 \begin{equation}\label{fexp}
x =a_{0}+\epsilon_{0}x_{0}=
\cdots=a_{0}+\frac{\epsilon_{0}}{a_{1}+\frac{\epsilon_{1}}{a_{2}   +\cdots  +\frac{\epsilon_{n-1}}{a_{n}+\epsilon_{n}x_{n}}   }         }= a_{\alpha,0}+\dfrac{\epsilon_{\alpha,0}}{a_{\alpha,1}+\dfrac{\epsilon_{\alpha,1}}{\ddots+\dfrac{\epsilon_{\alpha,n-1}}{a_{\alpha,n}+\dfrac{\epsilon_{\alpha,n}}{\ddots}}}}.
 \end{equation}
 We will denote the infinite $\alpha$-expansion of $x$ as $x=[(a_0,\epsilon_0);(a_{1},\epsilon_{1}),\cdots, (a_{n},\epsilon_{n}),\cdots].$ 
  Note that when $\alpha=1,$ we recover the standard continued fraction expansion defined by the iteration of the Gauss map. In that case, since $\epsilon_{n}=1$ for all $n,$ we simply omit it, recovering the classical notation of regular continued fractions. 
  When $\alpha=1/2,$ we obtain the so called nearest integer continued fraction; in this case $a_{n}\geq 2$ for all $n\geq1.$

 We shall often use the isomorphism between the group of real invertible $2\times2$ matrices and linear fractional transformations. Specifically, we define the action of a matrix on a real number $z$ as follows:
\begin{equation*}
\begin{pmatrix} a & b \\ c & d\end{pmatrix} \cdot z =\frac{az+b}{cz+d}.
\end{equation*}
Identifying matrices and fractional transformations is particularly convenient, since the product of matrices corresponds to the composition of the corresponding fractional transformations. Note also that, if $\varphi$ denotes the fractional transformation defined above, we can easily compute its derivative as:
$$\varphi'(z)=\frac{ad-bc}{(cz+d)^2}.$$
Thus, $\varphi$ is monotone increasing or decreasing accordingly to the sign of the  determinant of the corresponding matrix.

With this identification, the map $z\mapsto 1/(a+\epsilon z)$, which is the building block of $\alpha$-continued fractions, corresponds to the matrix $\begin{pmatrix} 0 & 1 \\ \epsilon & a\end{pmatrix}$ which has determinant $-\epsilon$.  If $S_n=((a_{0},\epsilon_{0});(a_{1},\epsilon_{1}),\cdots,(a_{n},\epsilon_{n}))$ are the first digits of the $\alpha$-expansion of a real value $x_0$, with a slight abuse of notation we identify $S_n$ with the product of the corresponding building blocks and define \footnote{Note that in this definition the action of the fractional transformation corresponding to the digit of index zero is different from all the others, and this is the reason why we use a semicolon (rather than a comma) to separate the digit of index zero from the rest of the expansion. When we consider the expansion of elements belonging to $I_\alpha$ we will often neglect  the  digit of index zero (simply omitting the first matrix): there hardly is any risk of confusion since in this case, the digit of index zero would correspond to the identity matrix.}
 \begin{equation}\label{stringaction}
S_n\cdot z=
\begin{pmatrix}
\epsilon_{0}\quad a_{0}\\
0 \quad 1
\end{pmatrix}
\begin{pmatrix}
{0}\quad {1}\\
\epsilon_{1} \quad a_{1}
\end{pmatrix}
\cdots
\begin{pmatrix}
{0}\quad {1}\\
\epsilon_{n} \quad a_{n}
\end{pmatrix}\cdot z.
\end{equation}
It is easy to prove by induction that 
$$S_n\cdot z=\begin{pmatrix}
{\epsilon_n p_{n-1}}\quad {p_n}\\
\epsilon_{n}q_{n-1} \quad q_{n}
\end{pmatrix}\cdot z, $$
 where the sequences $p_{n}$ and $q_{n}$ are recursively determined by the following  relation:
 \begin{equation*}
 p_{n}=a_{n}p_{n-1}+\epsilon_{n-1}p_{n-2},\quad q_{n}=a_{n}q_{n-1}+\epsilon_{n-1}q_{n-2},\quad p_{-1}={q_{-2}}=1,\quad p_{-2}=q_{-1}=0.\end{equation*}
The rational values 
 $\frac{p_{n}}{q_{n}}=S_n\cdot0=[ (a_0,\epsilon_0);(a_{1},\epsilon_{1}),\cdots,(a_{n},\epsilon_{n})            ]$ are called the {\em $\alpha$-convergents} of the value $x$.

By Binet Formula, we immediately obtain the following identity:
$$q_{n}p_{n-1}-p_{n}q_{n-1}=\epsilon_n det(S_n)=(-1)^{n}\epsilon_{0}\cdots\epsilon_{n-1} \in \{\pm 1\}.$$ 
Using \eqref{fexp}, we easily get
\begin{equation*}
x=S_n \cdot x_n = \frac{p_{n}+p_{n-1}\epsilon_{n}x_{n}}{q_{n}+q_{n-1}\epsilon_{n}x_{n}} \ \ \  \text{ hence  } \ \ \ 
x_{n}=-\epsilon_{n}\frac{q_{n}x-p_{n}}{q_{n-1}x-p_{n-1}}.
\end{equation*}
Thus for  $\beta_{n}=\prod^{n}_{i=0}x_{i}$ we get
 \begin{equation}
     \label{betan}
 \beta_{n}=\prod^{n}_{i=0}A^{i}_{\alpha}(x)=\prod^{n}_{i=0}x_{i}=(-1)^{n}(\epsilon_{1}\cdots\epsilon_{n-1})(q_{n}x-p_{n}) \text{ for } n\geq0, \text{ with } \beta_{-1}=1.
 \end{equation}

The following theorem summarizes the fundamental properties of the quantities $\beta_n$ we shall need later on:

 \begin{pro}[\cite{MaMoYo_97, MoCaMa_99}]\label{prop}
Let $\eta=\sup(\gamma,\sqrt{|1-2\alpha|}).$ Given $\alpha\in (0,1],$ for all $x\in\R\setminus\Q$ and for all $n\geq1$ one has 
 \begin{itemize}
 \item[\rm{(i)}] $q_{n+1}>q_{n}>0;$
 \item[\rm{(ii)}] $p_{n}>0$ when $x>0$ and $p_{n}<0$ when $x<0;$
 \item [\rm{(iii)}]$|q_{n}x-p_{n}|=\frac{1}{q_{n+1}+\epsilon_{n+1}q_{n}{x_{n+1}}}$ so that $\frac{1}{1+\alpha}<\beta_{n}q_{n+1}<\frac{1}{\alpha};$
 \item[\rm{(iv)}] if $g<\alpha\leq 1,$ $\beta_{n}\leq \bar{\alpha} g^{n};$
 \item [\rm{(v)}] if $0<\alpha\leq g,$ $\beta_{n}\leq \bar{\alpha} \eta^{n};$
  \end{itemize}
 \end{pro}

Note that the map $A_{\alpha}$ has countably many branches:
$$
 A_{\alpha}(x)=
 \begin{cases}
 \frac{1}{x}-k  &   \text{ for} \quad   \frac{1}{k+\alpha} <x  \leq \frac{1}{k},   \\
 k-\frac{1}{x} & \text{ for} \quad \frac{1}{k}<x\leq \frac{1}{k
 +\alpha-1}.
 \end{cases}
 $$
In fact the digits $(a_{i+1}, \epsilon_{i+1})$ determine the position of $x_i$ with respect to the partition determined by the branches of $A_\alpha$.

A sequence $ ( (a_{0},\epsilon_{0});(a_{1},\epsilon_{1}),\cdots)$ 
is called an \textit{admissible sequence} if the continued fraction $
    [   (a_{0},\epsilon_{0});(a_{1},\epsilon_{1}), \cdots ,(a_{n},\epsilon_{n}),\cdots,            ]$ represents the $\alpha$-continued fraction expansion of some irrational number $x\in I_{\alpha}.$ For any $n\geq 1,$ an \textit{admissible block} is defined as the finite truncation of an admissible sequence. Specifically, it  consists of the first $n+1$ terms of the admissible sequence expressed as  $S:=( (a_{0},\epsilon_{0});(a_{1},\epsilon_{1}),\cdots,(a_{n},\epsilon_{n})).$ We will also use the notation $|S|=n$.

For any $n\geq 1,$ we define the set $\mathcal L_{n}$ as follows     %
   $$ \mathcal L_{n}=\{[ (a_{0},\epsilon_{0})  ;(a_{1},\epsilon_{1}),\cdots,(a_{n},\epsilon_{n})            ]:   S=((a_{0},\epsilon_{0}); (a_{1},\epsilon_{1}),\cdots,(a_{n},\epsilon_{n}))
    \text{ is an admissible block} 
    \} $$


and let 
$$ \mathcal L=\bigcup_{n=1}^{\infty} \mathcal L_{n}.$$

Given any $n\geq 1$ and $S\in \mathcal L,$ we  define $I_{S}$ as the $n$-th cylinder generated by $S$, namely the set of all real numbers $x\in I_\alpha$ whose $\alpha$-continued fraction expansion begins with the string $S.$ 
Note that, for each admissible block $S$, the expression $S\cdot x$ represents the number obtained by appending the string $S$ to the beginning of the $\alpha$-continued fraction expansion of $x$; we denote this action by $\varphi_{S}$: $\varphi_{S}(x):=S\cdot x.$ Then $I_S=\varphi_S(J)$ for some interval $J\subset I_\alpha$ with non-empty interior. Note that $J$ need not coincide with $I_\alpha$, unless all branches of $A_\alpha$ are surjective (which is actually the case for $\alpha \in \{0,1/2,1\}$). 

When $x$ ranges inside a $n$--cylinder $I_S$, by equation \eqref{betan} the quantities $\beta_{n}, \beta_{n-1}$ are affine functions of $x$, and $x_n=\beta_n/\beta_{n-1}$; thus the general term  $\beta_{n-1}\log   \frac{1}{x_n}$ in the sum \eqref{aBFun} defining $B_\alpha$ has a form of the type $\omega(x)=A(x)[\log A(x)-\log B(x)]$ with $A,B$ affine functions of $x$, and we can easily check that this function is convex on its domain since its second derivative is always positive (on the domain of $\omega$):
\begin{equation}\label{generalterm}
    \omega'(x)=A'\log A-A'\log B+\frac{A'B-B'A}{B}, \ \ \ \ \ \omega''(x)=\frac{(A'B-B'A)^2}{AB^2}.
\end{equation}

If $r\in\Q$ then the $\alpha$-expansion of $r$ is finite: indeed  to determine the first partial quotient of $r\in\Q$
we write $r=a_{0}+\epsilon_{0}r_{0}$ with $r_{0}\in I_\alpha$, we can then compute iteratively the other partial quotients applying repeatedly the map $A_{\alpha}$ to $r_{0}:$ 
we get a (finite) sequence of rational values $A_{\alpha}^k(r_0)=\frac{|p_{k}|}{q_{k}}$ with
 $$|p_k|<q_k, \ \ \ \ \ \alpha-1 \leq \frac{q_{k-1}}{p_{k-1}}-a_k<\alpha, \ \ \ \  1\leq q_k=|p_{k-1}| <q_{k-1}.$$ 
 The last property shows that this algorithm will end in a finite number of steps. Indeed by finite descent we will eventually  get $q_k=1,$ $p_k=0,$ which means that $A_{\alpha}^k(r_0)= 0$.

Furthermore, as in the case $\alpha=1$, for all $\alpha\in (0,1),$  every $r\in\Q$ admits exactly two $\alpha$ continued fraction expansions i.e. 
\begin{lem}\label{cilindri}
Given $r\in\Q$ there exist two distinct admissible sequences $S, S'\in \mathcal L$ such that $S\cdot 0=r=S'\cdot0.$
Moreover, considering the following maps
$$ \varphi_{S}: y \mapsto S\cdot y  \; \; \mbox{ and }  \; \; \varphi_{S'}: y \mapsto S'\cdot y,$$
\begin{enumerate}
\item[\textnormal{(i)}] one is orientation preserving and the other is orientation reversing on a right neighbourhood of zero,
\item[ \textnormal{(ii)}] if $y>0$ is sufficiently small, all values of the form $S\cdot y $ (resp. $S'\cdot y $) have an $\alpha$-expansion starting with $S$ (resp. $S'$),
\item[\textnormal{(iii)}] $S$ and $S'$ can be used to parametrize the corresponding cylinders, i.e. there exist $\delta, \delta'>0$  such that
$$
I_S=\{S\cdot y , \,\,\, 0<y<\delta\} \,\,\,\,\,\,\,
I_{S'}=\{S'\cdot y , \,\,\, 0<y<\delta'\}.
$$
Hence $r$ is the separation point between the cylinders $I_S$ and $I_{S'}$.
\end{enumerate}

\end{lem}

\subsection{Behaviour of $B_{\alpha}$ near rational points.}

\begin{lem}\label{P1}
Let $\frac{p}{q}
\in\Q,$ then $
B_{\alpha}(x)\to\infty \text{ as }  x\to \frac{p}{q}$ for all $\alpha\in(0,1].$
\end{lem}

\begin{proof}
Let $S, S' \in \mathcal{L}$  be the two expansions of $\frac{p}{q}$ as in Lemma \ref{cilindri} such that $\frac{p}{q}=S\cdot0=S'\cdot0$ and $I_S\cup\frac{p}{q}\cup I_{S'}$ is a (punctured) neighbourhood of $\frac{p}{q}.$
Moreover,  suppose $x \in I_S,$ we  can write $x=S\cdot y$ where $y>0.$ Then
\begin{equation}\label{chh}
B_{\alpha}(x)\geq \beta_{n-1}(x)\log x_{n}^{-1}=\beta_{n-1} (S\cdot y)\log y^{-1}.
\end{equation}
Applying Lagrange intermediate value theorem we can find $\xi\in(0,y)$ such that
\begin{align*}
\left\vert x-\frac{p}{q} \right\vert=\left\vert S\cdot y-S\cdot 0\right\vert=\left\vert \varphi^{'}_{S}(\xi)\right\vert y\geq \min_{\xi\in(0,y)} \frac{1}{(q_{n}+q_{n-1}\epsilon_{n}\xi)^2}y\geq \frac{1}{4q_{n}^2}y.
\end{align*}
Therefore 
\begin{equation}\label{y}
y^{-1}\geq \frac{|x-\frac{p}{q}|^{-1}}{4q^2}
\end{equation}
 since  $q_{n}=q.$ By using \eqref{chh}, \eqref{y} and the fact that $\beta_{n-1}(S\cdot y)>\frac{1}{(\alpha+1)q}$ we obtain
\begin{equation*}
B_{\alpha}(x)\geq \frac{1}{(\alpha+1)q}(\log|x-\frac{p}{q}|^{-1}-\log(4q^2))\to\infty, \text{ as } x\to\frac{p}{q}.\end{equation*}

The same argument applies on $I_{S'}$ gaining the same result on the other half neighbourhood of $p/q$.
\end{proof}

\section{Lower semicontinuity of generalized Brjuno functions}
\label{SLSC}

In this section, we will discuss the lower semicontinuity of  $B_\alpha$   for different choices of $\alpha,$ since this property is necessary to repeat the argument of \cite{BaMa_12} for general $B_\alpha.$ We will prove that if $\alpha$ is rational, then $B_{\alpha}$ is lower semi-continuous. We also point out that the hypothesis $\alpha\in\Q$ cannot be dropped, since there exist irrational values $\alpha\in(1/2,1)$ for which $B_{\alpha}$ is not lower semi-continuous.

Using the lower semicontinuity, we also show that the intermediate value property holds for $B_{\alpha}$ when $\alpha$ is rational. This result is not necessary for our main result, but we believe it has its own interest.

We conclude this section by observing that the semicontinuity property can also be proved for more general Brjuno functions.

\subsection{Lower semicontinuity of $B_{\alpha}$ for $\alpha\in(0,1]\cap\Q.$}

\begin{lem}
For all $\alpha\in (0,1]\cap\Q$ the partial sum $B^{K}_{\alpha}(x)$ defined in \eqref{PS} is  infinitely differentiable (i.e. smooth) on every $K$-cylinder .
\end{lem}
\begin{proof}
Let $S\in \mathcal{L}_{K}$ and $\delta >0$ as in Lemma \ref{cilindri}, (iii). Since $\varphi_{S}:(0,\delta)\xrightarrow{ \sim} I_{S}$ is well defined,  it is sufficient  to show  that $B^{K}_{\alpha}(S\cdot y)$ is infinitely differentiable  in $y.$ For each $j$ such that $0\leq j\leq K,$  the term $A_\alpha^{j}(x)=(\sigma^{j}S)\cdot y$ is smooth in $y$; here, $\sigma^{j}$ denotes a shift map acting on $S$.  
Consequently,  $\beta_{j}(x)$ is also smooth in y. The sum $\sum^{K}_{j=0}\beta_{j-1}(S\cdot y)\log(1/(\sigma^{j}S)\cdot y)$ is  a finite sum of infinitely differentiable  functions. Therefore, it is also smooth in $y$ for $y\in(0,\delta).$
\end{proof}

\begin{pro}\label{LSCa}
Let $\alpha\in \Q.$ Then $B_{\alpha}$ is lower semi continuous.
\end{pro}
\begin{proof}
Let $D_{c}=\{x\in \R: B_{\alpha}(x)\leq c\}$ where $c>0$. To show that $B_{\alpha}(x)$ is lower semicontinuous we need to show $D_{c}$ is closed for all $c\in\R.$ Note that $B_{\alpha}(x)=\sup_{K\to\infty}B^{K}_{\alpha}(x)$ 
where $ B^{(K)}_{\alpha}(x)$ is defined in \eqref{PS}.

Therefore we can rewrite 
$D_{c}=\bigcap_{K\in\N}D_{K,c}$ where $D_{K,c}=\{x\in \R: B_{\alpha}^{(K)}(x)\leq c\}.$ Thus it is enough to show 
$D_{K,c}$ is closed for all $c$ and for all $K\in\N.$ Indeed if $(x_{n})_{n\in\N}\in D_{K,c}$ such that $x_{n}\to\bar{x}$ then 
by Lemma \ref{P1} $\bar{x}\notin \Q.$ That implies $A^{j}_{\alpha}$ is continuous at $\bar{x}$ for all $1\leq j\leq K.$ Hence $B^{K}_{\alpha}$ is continuous at $\bar{x}$ and $\bar{x}\in I_{S}$ for some $K$-cylinder. Since $I_{S}$ is open and $ x_{n}\to\bar{x}$ there exist $n_{0}$ such that $x_{n}\in I_{S}$ for every $n\geq n_{0}.$ By the fact that $x, \bar{x}\in I_{S}$ we can write $\bar{x}=S\cdot\bar{y},$ $x_{n}=S\cdot y_{n},$  with $y_{n}\to\bar{y}$ for all $n\geq n_{0}.$ By the continuity of $B^{K}_{\alpha}$ at $\bar{x}$ it follows that $B^{K}_{\alpha}(S\cdot y_{n})\to B^{K}_{\alpha}(S\cdot \bar{y})$ as $n\to\infty.$ Hence $B^{K}_{\alpha}(x_{n})\to B^{K}_{\alpha}(\bar{x})$ and $B^{K}_{\alpha}(\bar{x})\leq c$ which implies $\bar{x}\in D_{K,c}.$ Thus $D_{K,c}$ is closed and  consequently $D_{c}$ is closed.
\end{proof}

Since $B_\alpha$ is lower semicontinuous and 1-periodic, it can be considered as a function on the circle (which is compact), hence admits an absolute minimum. Without loss of generality, we can think that this minimum belongs to the period $[\alpha-1,\alpha)$, and by the symmetry on $[\alpha-1,1-\alpha]$ one can find a global minimum on $(0, \bar{\alpha})$.
\begin{cor}
If $\alpha$ is rational, the Brjuno function $B_{\alpha}$ has a global  minimum on $[0,\bar{\alpha}].$
\end{cor}

\subsection{Remark on the lower semicontinuity of Brjuno function when $\alpha$ is irrational.}\label{lscfail}

 It is worth pointing out that the hypothesis $\alpha\in (0,1]\cap\mathbb{Q}$ in Proposition \ref{LSCa} is not a technical assumption. Indeed, there exist irrational values of $\alpha$ for which the Brjuno function $B_\alpha$ {\bf is not} lower semi-continuous, as shown by the the following example (see also Figure \ref{due}).
 \begin{exa}
 Suppose $\hat{\alpha}:=\frac{1}{1+\frac{1}{a+g}}$ where $a\geq2$ is a positive integer. Clearly $\hat{\alpha}\in[0,1]\setminus\Q.$ Then
 \begin{enumerate}
 \item[\textnormal{(i)}] $B_{\hat{\alpha}}(\hat{\alpha})=B_{\hat{\alpha}}(1-\hat{\alpha}) =\log(a+1+g)+\frac{1}{a+1+g}B_{\hat{\alpha}}(g).$
 \item[\textnormal{(ii)}] $\lim\inf_{x\to\ah^{+}} B_{\hat{\alpha}}(x)=B_{\hat{\alpha}}(\hat{\alpha}).$
 \item[\textnormal{(iii)}] $\lim\inf_{x\to\ah^{-}} B_{\ah}(x)<B_{\ah}(\ah).$
  \end{enumerate}
 \end{exa}

    \begin{figure}[h]
 \centering
 \includegraphics[width=0.5\textwidth]{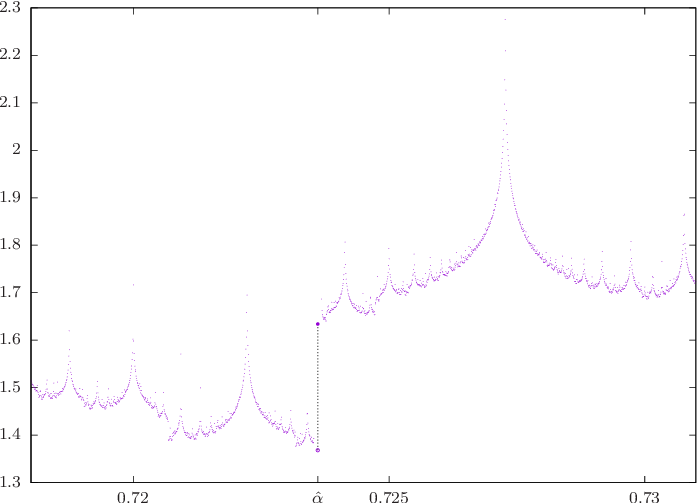}
 \caption{The graph showing lower semicontinuity fails for some irrational $\hat{\alpha}$; here a=2 and $\hat{\alpha}=\frac{2+g}{3+g}.$}
 \label{due}
 \end{figure}

\begin{proof}
\begin{enumerate}
 \item[(i)] Clearly, follows from the definition \eqref{BFE} and the fact that $A_{\ah}(1-\ah)=g.$
  \item[(ii)] 
  Suppose $x\geq \ah.$ Then we can write $x=\frac{1}{1+\frac{1}{a+g+\epsilon}}$ and $1-x=\frac{1}{a+1+g+\epsilon}$ where $\epsilon>0.$ Therefore
  \begin{align*}
  \liminf_{x\to\ah^{+}}B_{\ah}(x)&=\lim\inf_{\epsilon\to0^{+}}\left(\log(a+1+g+\epsilon)+\frac{1}{a+1+g+\epsilon}\log B_{\ah}(g+\epsilon)\right)\\
  &=\log(a+1+g)+\frac{1}{a+1+g}\lim\inf_{\epsilon\to0^{+}}\log B_{\ah}(g+\epsilon),
    \end{align*}
  and the required results follows since $\lim\inf_{\epsilon\to0^{+}}\log B_{\ah}(g+\epsilon)=B_{\alpha}(g).$

 \item[(iii)] 
  Suppose $x\leq \ah$  and write $x=\frac{1}{1+\frac{1}{a+g-\epsilon}}$ where $\epsilon>0.$ Therefore
  \begin{align*}
  \liminf_{x\to\ah^{-}}B_{\ah}(x)&=\lim\inf_{\epsilon\to0^{-}}\left(\log(a+1+g-\epsilon)-\log(a+g-\epsilon)+x B_{\ah}(\frac{1}{a+g-\epsilon})\right)\\
  &=\lim\inf_{\epsilon\to0^{-}}\left(\log(a+1+g-\epsilon)-\log(a+g-\epsilon)+x\left( \log(a+g-\epsilon)+\frac{1}{a+g-\epsilon}B_{\ah}(g)\right) \right)  \\
  &=B_{\ah}+\log(a+g)[-1+\ah]<B_{\ah}(\ah),
      \end{align*}
  the last inequality follows by the fact that $\log(a+g)[-1+\ah]<0.$
 
  \end{enumerate}\end{proof}

\subsection{Intermediate value property for $B_{\alpha}$ when $\alpha\in (0,1]\cap\mathbb{Q}.$}
Exploiting  the lower semicontinuity of $B_\alpha$ (which holds  when $\alpha\in \mathbb Q$), one can prove that it satisfies the intermediate value property, (even if it  is not a continuous function).
\begin{pro}\label{mean}[Intermediate Value Theorem for $B_{\alpha}$]
Let $\alpha\in\Q.$ If $B_{\alpha}(a)=\lambda,$ $B_{\alpha}(b)=\nu,$ then for all $\rho$ between $\lambda$ and $\nu$
there exists $\xi\in(a,b)$ such that $B_{\alpha}(\xi)=\rho.$
\end{pro}

The property is interesting in itself, but before giving a proof of the above proposition, let us state (and prove) the following somewhat surprising consequence:

\begin{cor}
Let $\alpha\in\Q\cap (0,1).$ Then the closure of the graph of $B_\alpha$ fills the region above it:
\begin{equation*}
 \overline{\left\{ (x,y)\in\R^{2}: y=B_{\alpha}(x) \right\}} = \left\{ (x,y)\in\R^{2}: y\geq B_{\alpha}(x) \right\}.
\end{equation*}
\end{cor}

\begin{proof}
Let $B_{\alpha}(\xi)<+\infty,$ $\rho> B(\xi).$ For any $n\in \mathbb{N}$ we can pick a rational value $r$
such that
$\xi<r<\xi+\frac{1}{n},$  $B_{\alpha}(r)=+\infty.$ This implies there exists $\xi_n \in (\xi,\xi+\frac{1}{n}),$ and
$B_{\alpha}(\xi_{n})=\rho,$ so that  
 $\left( \xi_{n}, B(\xi_{n})\right) \to (\xi,\rho)$ as $n\to\infty$.
\end{proof}

The fact that the intermediate value property holds is not too surprising, indeed it holds for the related Brjuno function $\tilde{B},$ defined by
$$\tilde{B}(x)=\sum^{\infty}_{k=0}\frac{\log a_{k+1}}{q_{k}},$$ 
where 
$x=[a_{0}; a_{1}, a_{2},\cdots,a_{n},a_{n+1},\cdots]$ is the regular continued fraction expansion of $x$ (i.e., the one associated to the Gauss map) and $p_n/q_n=[a_{0}; a_{1}, a_{2},\cdots,a_{n}]$ is the n-th convergent of $x$ 
(see \cite[Lemma 4-(i)]{MaCa08} which plays the role of Corollary \ref{level}).

Let us first provide some useful remarks and a lemma, which will be useful for the proof of Proposition \ref{mean} (which takes the last few lines at the end of this section).

\begin{rem}
\begin{enumerate}
\item[\textnormal{1.}]For all $\alpha\in (0,1]$ there is $H>0$ such that $|B_\alpha(x)-\tilde{B}(x)|\leq H$ for all $x$.
\item[\textnormal{2. }]If $x=[0; a_{1}, a_{2},\cdots,a_{n},a_{n+1},\cdots]$  then for all $N\in \mathbb{N}$ there is some $K=K(N)$ such that
$$A^N_{\alpha}(x)=[0; b,b', a_K, a_{K+1}, a_{K+2}, ...].$$
This means that the continued fraction expansion of $A^N_{\alpha}(x)$ coincides with a tail of the continued fraction of $x$ \textnormal{(}except possibly the first two partial quotients\textnormal{)}.
\end{enumerate}
\end{rem}

The proof of the first remark is easy, since it is well known that both $B_\alpha - B_1$ and $B_1-\tilde{B}$ are bounded functions. The proof of the second remark is easily proved by induction; indeed, the claim is clearly true for $N=0$, and if the claim holds for a certain integer $N$ we have 
 $$A^N_{\alpha}(x)=[0; b,b', a_K, a_{K+1}, a_{K+2}, ...], \ \ \ A_1(A^N_{\alpha}(x))=[0; b', a_K, a_{K+1}, a_{K+2}, ...]$$
Thus, if $A_1(A^N_{\alpha}(x)) \in [0, \bar{\alpha})$ then $A^{N+1}_{\alpha}(x))=A_1(A^N_{\alpha}(x))$, otherwise
$$A^{N+1}_{\alpha}(x))=1-A_1(A^N_{\alpha}(x))=
\left\{
\begin{array}{ll}
[0;1,b'-1,a_K, a_{K+1}, a_{K+2}, ...] & \mbox{if } b'>1 \\ 
{[0;1+a_K, a_{K+1}, a_{K+2}, ...]} & \mbox{if } b'=1.
\end{array}
\right.
$$
In any case , it is immediate to check that and the claim is true for $N+1$ as well.

\begin{lem}\label{lm1}
Let $B$ denote either $\tilde{B}$ or $B_{\alpha}$ for $\alpha \in (0,1]\cap \mathbb{Q}$.\\
Let $\xi=[a_0;a_1,a_2,...,a_{n-1},a_n, a_{n+1}, ...]$ be such that $B(\xi)<+\infty$ and let 
$\xi_n =[a_0;a_1,a_2,...,a_{n-1},1+a_n,1,1,1, ...]$. Then
\begin{enumerate} 
\item[\textnormal{(i)}] $\lim_{n\to +\infty} \xi_n=\xi$ and $\lim_{n\to +\infty} B(\xi_n)=B(\xi)$;
\item[\textnormal{(ii)}] for all $\epsilon>0$ there exists $\xi^{+}\in (\xi, \xi+\epsilon)$ and $\xi^{-}\in (\xi-\epsilon, \xi)$ such that 
\begin{equation*}
\left| B(\xi^{\pm}) -B(\xi)  \right|<\epsilon.
\end{equation*}
\end{enumerate}
\end{lem}

\begin{proof}
 We immediately check that $\xi_n\neq \xi$ for all $n$, $\xi_n  \to \xi$ as $n\to +\infty$ and ${\rm sign}(\xi_n-\xi)=(-1)^n$. It is therefore clear that to prove claim (ii) it is enough to prove claim (i).

We first prove claim (i) for $B=\tilde{B}$. In this case, it is immediate to check that $q_k(\xi)=q_k(\xi_n)$ for all $k \leq n-1$, hence the first $(n-1)$ terms of the sums cancel out and
$$ |B(\xi)-B(\xi_n)|\leq \left|\frac{\log(a_n)-\log(1+a_n)}{q_{n-1}}\right| + \sum_{k=n}^{+\infty}\frac{\log a_{k+1}}{q_k(\xi)} \leq \frac{\log 2}{q_{n-1}}+\sum_{k=n}^{+\infty}\frac{\log a_{k+1}}{q_k(\xi)}$$
Both items on the right in the above formula vanish when $n\to +\infty$, proving (i) for $\tilde{B}$.
 
To prove (i) for $B=B_\alpha$ we shall check that for all $\epsilon>0$ there exists some $\bar{n}$ such that $|B_\alpha(\xi_n)-B_\alpha(\xi)|<\epsilon$ for all $n\geq \bar{n}$.

Given any $\epsilon > 0$ let us first fix $N$ such that $\beta_N(x)<\frac{\epsilon}{4H} \ \forall x$ (this can be done by (iv) or (v) of Proposition \ref{prop}), and let us write

$$
\begin{array}{lll}
\left| B_{\alpha}(\xi) -B_{\alpha}(\xi_n)  \right|  & \leq & 
\left| B^{	N}_{\alpha}(\xi) -B^{N}_{\alpha}(\xi_n) \right| +
\left|   \beta_{N}(\xi) B_{\alpha} (A_{\alpha}^{N+1}(\xi)) -\beta_{N}(\xi_n)   B_{\alpha} (A_{\alpha}^{N+1}(\xi_n))     \right|\\
&\leq& \left| B^{	N}_{\alpha}(\xi) -B^{N}_{\alpha}(\xi_n) \right| +
|   \beta_{N}(\xi)-\beta_{N}(\xi_n)| | B_{\alpha} (A_{\alpha}^{N+1}(\xi))| + \\
&& + \beta_{N}(\xi_n)|B_{\alpha} (A_{\alpha}^{N+1}(\xi_n))-B_{\alpha} (A_{\alpha}^{N+1}(\xi))|
\end{array}
$$

Since $B_{\alpha}(\xi)<+\infty$ the point $\xi$ is irrational, and since $\alpha\in \mathbb{Q}$ there exists a neighbourhood $U$ of the point $\xi$ such that the
 truncated Brjuno function $B^N_\alpha$ (defined in \eqref{PS}), $\beta_n$, and $A_\alpha^{N+1}$ are all continuous on $U$, hence 
$$\begin{array}{r}
\left| B^{	N}_{\alpha}(\xi) -B^{N}_{\alpha}(\xi_n) \right|=o(1)\\
|   \beta_{N}(\xi)-\beta_{N}(\xi_n)| =o(1)
\end{array}  \ \ \mbox{for} \ n\to +\infty$$

On the other hand 
$$
\begin{array}{lll}
|B_{\alpha} (A_{\alpha}^{N+1}(\xi_n))-B_{\alpha} (A_{\alpha}^{N+1}(\xi))| &\leq & 
|B_{\alpha} (A_{\alpha}^{N+1}(\xi_n))- \tilde{B} (A_{\alpha}^{N+1}(\xi_n))| +
|\tilde{B} (A_{\alpha}^{N+1}(\xi_n))-\tilde{B} (A_{\alpha}^{N+1}(\xi))|+\\
&&+
 |\tilde{B} (A_{\alpha}^{N+1}(\xi))-  B_{\alpha} (A_{\alpha}^{N+1}(\xi))|\\
 &\leq & 2H +|\tilde{B} (A_{\alpha}^{N+1}(\xi_n))-\tilde{B} (A_{\alpha}^{N+1}(\xi))|.
\end{array}  
$$

Since we are assuming that $\xi_n \in U$ we will have that
$$A_{\alpha}^{N+1}(\xi)=[0;b,b', a_K, ..., a_{n-1}, a_n, a_{n+1}, ...], \ \ \ \
A_{\alpha}^{N+1}(\xi_n)=[0;b,b', a_K, ..., a_{n-1}, 1+a_n, 1,1,1, ...]
$$
and thus, using the fact that property (i) holds for $\tilde{B}$, we get $|\tilde{B} (A_{\alpha}^{N+1}(\xi_n))-\tilde{B} (A_{\alpha}^{N+1}(\xi))|= o(1)$ as $n\to +\infty$.

Summing up altogether we deduce that
$$ \left| B_{\alpha}(\xi) -B_{\alpha}(\xi_n)  \right|\leq o(1) +\frac{\epsilon }{4H} (2H +o(1)) \ \ \ \mbox{for } n\to +\infty$$
and since the right hand of the above formula tends to $\epsilon/2$ as $n\to +\infty$,  we get that there is $\bar{n}$ such that $| B_{\alpha}(\xi) -B_{\alpha}(\xi_n) |<\epsilon$ for all $n>\bar{n}$.
\end{proof}

\begin{cor}\label{level}
Let $\alpha \in (0,1]\cap \mathbb{Q}$, $\rho >0$ and let $V=(\xi_0,\xi_1)$ be a nonempty connected component of the open set $A_\rho:=\{ x : B_\alpha (x)<\rho\}$. Then $B_\alpha (\xi_0)=B_\alpha (\xi_1)=\rho$.
\end{cor}

\begin{proof}
The set $A_\rho$ is open because $B_\alpha $ is lower semicontinuous, thus since $V$ is a connected component of $A_\rho$ then $\partial V \cap A_\rho=\emptyset$, and this means that $B_\alpha (\xi_i) \leq \rho$ ($i\in \{0,1\}$). If by contradiction $B_\alpha (\xi_0) < \rho$, then by Lemma \ref{lm1} for all $\epsilon>0$ it would be possible to find $\xi^+\in (\xi_0, \xi_0+\epsilon)$ such that $B_\alpha (\xi^+)< B_\alpha (\xi_0)+\epsilon$, and if we choose $0<\epsilon<\rho- B_\alpha (\xi_0)$ we get that $\xi^+ \in (\xi_0,\xi_1)$ but $B_\alpha (\xi^+)<\rho$, a contradiction. Hence  $B_\alpha (\xi_0)=\rho.$ The same kind of argument proves that $B_\alpha (\xi_1)=\rho$ as well.
\end{proof}

Now we can easily prove Proposition \ref{mean}: indeed assume that $B_\alpha(a)=\lambda >\nu = B_\alpha(b)$ and let us fix any $\rho \in (\nu,\lambda)$. Then if $V=(\xi_0,\xi_1)$ is the connected component of $A_\rho$ containing $a$ we have that $a<\xi_1<b$ and $B(\xi_1)=\rho$ as claimed. The same argument works in case $\lambda<\nu$.

 \subsection{On the lower semicontinuity of other more general versions of Brjuno function.}    
 
 Unlike the results of Section \ref{Ca1} and the following sections, the argument to prove lower semicontinuity is quite general and with minor modifications it applies also to other variants of the Brjuno function.
 
Let $u:(0,1)\to\R^{
+}$ be a positive $C^{1}$ function such that $\lim_{x\to0^{+}}u(x)=\infty$ and $\nu\in\N$ be fixed. One can then define the following class of generalized Brjuno functions which includes those studied in \cite{LuMar_10},
\begin{equation}\label{GBF}
B_{\alpha,\nu,u}(x)=\sum^{\infty}_{j=0} \beta^{\nu}_{j-1}(x)u(x_{j}).
\end{equation}

Definition \eqref{GBF} is more general, as for different choices of $\nu$ and $u$ it implies various classical Brjuno functions. For example if we take $\nu=1,$ $u(x)=-\log (x)$ it implies classical $\alpha$-Brjuno function as defined in \eqref{aBFun}. For other choices of the singular behaviour of $u$ at zero, the condition $B_{\alpha,\nu,u}<\infty$ leads to different Diophantine conditions. For instance, if we choose  $u(x)=x^{-1/\sigma},$ where $\sigma>2,$ it is not difficult to check that if $B_{1,1,\sigma}(x)=\sum^{\infty}_{j=0} \beta_{j-1}(x)x_{j}^{-1/\sigma}<\infty,$ then $x$ is Diophantine number i.e., $x\in$ $CD(\sigma):=\{x\in\R\setminus\Q: q_{n+1}=O(q^{1+\sigma}_{n})\}.$ For more details regarding the function $B_{1,1,\sigma}$ we refer the reader to \cite{MaMoYo_06,LuMar_10}.

\begin{rem}\label{ch}
 Let $B_{\alpha,u,\nu}(x)$ be as defined in \eqref{GBF}. Then $B_{\alpha,u,\nu}(x)\to\infty$ as $x\to\frac{p}{q}.$
\end{rem}
 The proof is on the similar lines as of Lemma \ref{P1}. Indeed, we have $B_{\alpha,u,\nu}(x)=\sum^{\infty}_{j=0} \beta^{\nu}_{j-1}(x)u(x_{j})\ge \beta^{\nu}_{n-1}(x) u(x_{n})=\beta^{\nu}_{n-1} (S\cdot y)u( y).$  By using the value of $y$ from \eqref{y}, we know that $y\to 0$ as $x\to\frac{p}{q},$ and by the definition of $u,$ we have $u(y)\to\infty.$ Consequently $B_{\alpha,u,\nu}(x)\to\infty.$ Since $\nu$ is fixed positive integer, the term $\beta^{\nu}_{n-1}\geq \frac{1}{(\alpha+1)^{\nu}q^{\nu}_{n}}$ becomes smaller and smaller but remains non-zero.       

By Remark \ref{ch} and the fact that the partial sum $B^{K}_{\alpha,u,\nu}(x)$ is smooth on every $K$-cylinder, we have the following.
\begin{rem}
The function $B_{\alpha,u,\nu}$ is lower semicontinuous for all $\alpha\in \Q \cap[0,1].$
\end{rem}

\section{Scaling properties and local minima of $B_{1}$}\label{Ca1}
Throughout this section, we will focus on the classical case $\alpha=1.$ In view of the result  from \cite{BaMa_20} by Balazard-Martin, it is natural to ask about the minima on other intervals. For instance using, Theorem \ref{BM}, it is not difficult to prove the following:
 \begin{cor}\label{CBM}
 $\min_{x\in (0,1/2)}B_{1}(x)\geq B_{1}(1-g)=B_{1}(g^2).$
  \end{cor}
  
 The proof of Corollary \ref{CBM} is a routine computation, for completeness, we include its proof in Subsection \ref{PCBM}.
 
 Note that both $g$ and $g^{2}$ are ``noble'' numbers in the sense of definition \ref{noble}. Specifically a number $\nu$ is noble if and only if its regular continued fraction  expansion has the same tail as $g$:
 \begin{equation*}
 \nu=[0;a_{1},a_{2},\cdots,a_{n},\bar{1}].  
 \end{equation*}
 Equivalently, using the notation introduced by equation \eqref{stringaction} in Section 2, $\nu$ is noble if there exists a sequence $S=(a_1,...,a_n)$ such that $\nu=S \cdot g$.

\subsection{Proof of Theorem \ref{AIM}}

The proof of Theorem \ref{AIM} relies on some scaling properties of the function $B_1$ near its absolute minimum $g.$ Lemma \ref{L1} and Lemma \ref{L2} will give a precise description of this scaling property, which is actually quite evident from numerical data (see also Figure \ref{f:scaling}). 
 
Let $\Phi(x)=\frac{1}{1+x},$ and for all $n\geq 1,$ let us define the recursive relation: 
\begin{equation}\label{RRR}
x_{n+1}=\Phi(x_{n})   \quad \text{ where } x_{0}\in (0,1/2).
\end{equation}
Note that $B_{1}(g)=\frac{\log 1/g}{1-g}.$

In Lemma \ref{L1}, we prove the scaling property of $\mathcal E _{n} :=B_{1}(x_{n})-B_{1}(g),$ where the initial value $\mathcal E _{0}$ satisfies $\mathcal E _{0}\geq \left(3-\frac{1}{1-g}\right)\log \frac{1}{g}>0.$ 

\begin{lem}\label{L1}
\begin{align*}
\mathcal E _{2n+1}\geq \sigma \mathcal E _{1} g^{2n} \quad \text{ where }\quad \sigma:= \exp (-\sum^{\infty}_{n=1}\log (g^2 \frac{F_{2n+2}}{F_{2n}}))
\end{align*}
where $F_{n}$ are Fibonacci numbers with $F_{-1}=1$ and $F_{0}=0.$

\end{lem}

\begin{figure}[h]
 \centering
 \includegraphics[width=0.8\textwidth]{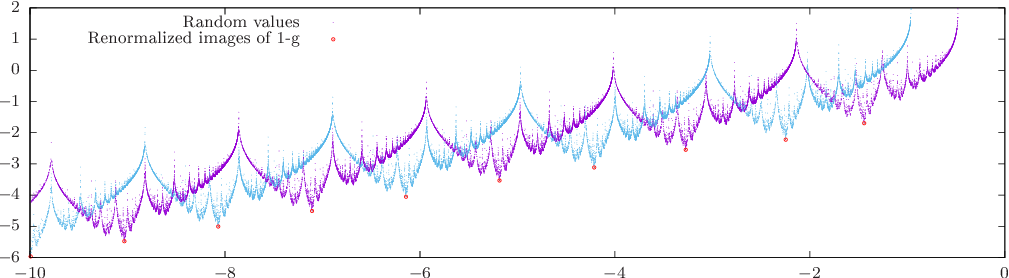}
 \caption{ Graph of $\left(\log|x-g|,\log (B_{1}(x)-B_{1}(g))\right)$ for various values of $x$ including $x_n=\Phi^{n}(1-g)$; the dots are represented in two different colours accordingly to the sign of $x-g$. }
 \label{f:scaling}
 \end{figure} 

\begin{proof}
 Since $x_{1}=\frac{1}{1+x_{0}},x_{2}=\frac{1+x_{0}}{2+x_{0}},\cdots,$ continuing in this way, \eqref{RRR} induces the recursive relation
\begin{equation}\label{F}
x_{n}=\frac{F_{n}+x_{0}F_{n-1}}{F_{n+1}+x_{0}F_{n}}, \quad \forall {n\geq 1} \text{ and } x_{0}\in(0,1/2).
\end{equation}
Using \eqref{BFE}, we write
\begin{align*}
B_{1}(\Phi(x_{n}))&=-\log(\Phi(x_{n}))+\Phi(x_{n})B_{1}(x_{n}) \text{ and}\\
B_{1}(\Phi(g))&=-\log(\Phi(g))+\Phi(g)B_{1}(g).
\end{align*}
Using \eqref{RRR} and the fact that $\Phi(g)=g,$ we have
\begin{align*}
B_{1}(x_{n+1})-B_{1}(g)=-\log\frac{x_{n+1}}{g}+x_{n+1}[B_{1}(x_{n})-B_{1}(g)]+B_{1}(g)[x_{n+1}-g].
\end{align*}
Setting $\mathcal E_{n}:=B_{1}(x_{n})-B_{1}(g)$ and $\delta_{n}:=x_{n}-g$ and observing that $-\log t\geq 1-t,$ we get 
\begin{equation*}
\mathcal E_{n+1}\geq x_{n+1}\mathcal E_{n}-l \delta_{n+1}
\end{equation*}
where $l:=\frac{1}{g}-B_{1}(g)$ and $l> 0.$
Hence
\begin{align*}
\mathcal E_{n+2}&\geq x_{n+2}\mathcal E_{n+1}-l \delta_{n+2} \geq x_{n+2}x_{n+1}\mathcal E_{n}-l (x_{n+2}\delta_{n+1}+\delta_{n+2}).
\end{align*}
Note that 
\begin{align*}
x_{n+2}\delta_{n+1}+\delta_{n+2}=x_{n+2}(1+x_{n+1})-g(1+x_{n+2})=1-g(1+x_{n+2}),
\end{align*}
which is positive if and only if $x_{n+2}<g$ i.e when $n$ is even.
Therefore for all $n\ge 1,$ 
\begin{align*}
\mathcal E_{2n+1}\geq \lambda_{n} \mathcal E_{2n-1},
\end{align*}
where 
$\lambda_{n}:=x_{2n+1}x_{2n}$ and by using \eqref{F}, we can write
\begin{equation*}
\lambda_{n}=\frac{F_{2n+1}+x_{0}F_{2n-1}}{F_{2n+2}+x_{0}F_{2n+1}}.
\end{equation*} Further,
for all $n\ge1,$ we have the following bounds for $\lambda_{n}$
\begin{equation*}
\frac{F_{2n}}{F_{2n+2}}\leq \lambda_{n} \leq \frac{F_{2n+2}}{F_{2n+4}},
\end{equation*}
i.e. $\lambda_{n}\geq 1/3$ and $\lambda_{n}\to g^2$ as $n\to\infty$ (exponentially).

In order to obtain the optimal estimate, first note that 
\begin{align*}
\mathcal E_{2n+1}&=\big( \prod^n_{k=1}\lambda_{k} \big) \mathcal E_{1} \text{ with} \\
\prod^n_{k=1}\lambda_{k} &=\exp (\sum^{n}_{k=1}\log \lambda_{k})=\exp (n\log g^{2}+\sum^{n}_{k=1}\log \frac{\lambda_{k}}{g^2})\\
&=g^{2n}\exp (\sum^{n}_{k=1}\log \frac{\lambda_{k}}{g^2}).
\end{align*}
The term $\sum^{\infty}_{k=1}\log \frac{\lambda_{k}}{g^2}$ in last equation is an absolutely convergent series because $\lambda_{k}$  converges to $g^{2}$ at an exponential rate.

Therefore, we can write
\begin{align}\label{E2}
\mathcal E _{2n+1}\geq \sigma \mathcal E _{1} g^{2n} \quad \text{ where }\quad \sigma:= \exp (-\sum^{\infty}_{n=1}\log (g^2 \frac{F_{2n+2}}{F_{2n}})).
\end{align}
 \end{proof}
The estimate \eqref{E2} implies that $B_{1}$ has a cusp like minima at $g,$ namely 
\begin{equation*}
B_{1}(x)-B_{1}(g)\geq c(x-g)^{1/2}        \quad \text{ with }
\end{equation*} 
(where the value of the exponent $1/2$ is obtained by the corresponding estimate and by comparing with the fact that 
$\delta_{n}\asymp g^{2n}$ for any $n\geq1$).

\begin{lem}\label{L2}
\begin{equation*}
B_{1}(x)-B_{1}(g) \geq \frac{\sigma \mathcal E_{1}}{2} {(x-g)^{1/2}} \text{ for } x>g.
\end{equation*}
\end{lem}
\begin{proof}
 Define $t_{n}:=\Phi^{2n-1}(\frac{1}{2})=\frac{F_{2n+1}}{F_{2n+2}}$ such that $t_{0}:=1,$ $t_{1}:=\frac{2}{3},$ $t_{2}:=\frac{5}{8} \cdots$
and $I_{n}=\Phi^{2n-1}((0,\frac{1}{2}))=(t_{n},t_n-1).$
Observe that
\begin{equation*}
t_{n}-g=\frac{F_{2n+1}}{F_{2n+2}}-g=g^{4n+3}\frac{1+g^{2}}{1-g^{4n+4}}
\end{equation*}
and 
\begin{equation*}
1\leq \frac{1+g^2}{1-g^{4n+4}}\leq G.
\end{equation*}

On the other hand, by \eqref{E2} we get 
$x\in I_{n}\implies B_{1}(x)-B_{1}(g)\geq \sigma\mathcal E _{1}g^{2n}.$ Therefore, if 
$t_{n}\leq x \leq t_{n-1}$ then

\begin{align*}
x-g\leq t_{n-1} -g \leq g^{4n-1}G=G^{2}g^{4n} \text{ and }
\end{align*}
\begin{align*}
B_{1}(x)-B_{1}(g)\geq \sigma \mathcal E_{1} g^{2n}=\frac{\sigma\mathcal E_{1}}{G}Gg^{2n}.
\end{align*}
Hence for all $x\in I_{n},$ and for all $ n\geq 1$ 
\begin{equation*}B_{1}(x)-B_{1}(g)\geq \frac{\sigma\mathcal E_{1}}{G}({x-g})^{1/2}.
\end{equation*}
Since $\bigcup_{n\geq 1}I_{n}=(g,1),$ we obtain the required result, with a constant which is slightly better.
\end{proof} 
 
 To complete the proof of Theorem \ref{AIM} we will also include the case when $1/2<x<g.$
 \begin{lem}
 \begin{equation*}
B_{1}(x)-B_{1}(g)\geq \frac{\sigma\mathcal E_{1}}{2}\sqrt{g-x} \text{ for } 1/2<x<g.
\end{equation*}
 \end{lem}
 \begin{proof}
 Let $1/2<x<g$ and write $x=\Phi(t)$ with $t\in(g,1).$ Then by repeating the same argument used in the proof of last lemma  we obtain
\begin{align*}
B_{1}(\Phi(t))-B_{1}(\Phi(g))&\geq (\Phi(t)-\Phi(g))\left(B_{1}(g)-\frac{1}{g}\right)+\Phi(t)(B_{1}(t)-B_{1}(g))\\
&\geq \Phi(t)\frac{\sigma\mathcal E_{1}}{G} \sqrt{t-g}.
\end{align*}
Observe that 
\begin{align*}
\Phi(g)-\Phi(t)&=\Phi'(\xi)(g-t) \text{ for some } \xi\in(g,t)\\
&=\frac{1}{(1+\xi)^2} (t-g)     \text{ since } \Phi'(\xi)=-\frac{1}{(1+\xi)^2}.
\end{align*}
Hence $\sqrt{t-g}=(1+\xi)\sqrt{g-x}$ for some $\xi\in(g,t).$
Thus, we have the inequality
\begin{align*}
B_{1}(x)-B_{1}(g)\geq \frac{1+\xi}{1+t} \frac{\sigma \mathcal E_{1}}{G}\sqrt{g-x}
\end{align*}
with $t\in(g,1)$ and $\xi\in(g,t),$ so that 
\begin{equation*}
\frac{1+\xi}{1+t}\geq \frac{1+g}{1+t}\geq\frac{1+g}{2}=\frac{G}{2}.
\end{equation*}
This implies
\begin{equation*}
B_{1}(x)-B_{1}(g)\geq \frac{\sigma\mathcal E_{1}}{2}\sqrt{g-x} \text{ for } 1/2<x<g, 
\end{equation*}
hence proving the claim. 
 \end{proof}

The theorem we just proved is very technical, but has a quite interesting consequence: the property of being a local minimum propagates from $g$ to all other noble numbers:
 \begin{cor}
 Let $\nu$ be a noble number.  Then $\nu$ is a local minimum of $B_{1}.$
 \end{cor}
 \begin{proof}
 Let $\nu= S\cdot g$ with $|S|=K$, and let us use the map $x \mapsto S \cdot x$ to parametrize a neighbourhood of $\nu.$ We have then
 \begin{align*}
 B_{1}(S\cdot x)-B_{1}(S\cdot g)&=B^{(K)}_{1} (S\cdot x)-B^{(K)}_{1} (S\cdot g)+(S\cdot x) B_{1}(x)+(S\cdot g) B_{1}(g)  \\
 &\geq  B^{(K)}_{1} (S\cdot x)-B^{(K)}_{1} (S\cdot g)+(S\cdot x -S\cdot g) B_{1}(g) +
 c(S\cdot x)|x-g|^{1/2}.
  \end{align*}
  Let us denote the right hand side of last inequality by $\phi(x).$ It is easy to see that there exists an open neighbourhood $U$ of $g$  such that $\phi$ is continuous on $U$ and  differentiable on $U\setminus\{g\},$ and $\lim_{x \to g^\pm} \phi'(x)=\pm \infty$. Therefore, $\phi$ has a minimum at $g,$ which implies $B_1$ has a local minimum at $\nu=S\cdot g.$
  \end{proof}
 
 \subsection{Proof of Corollary \ref{CBM}} \label{PCBM}
\begin{proof}

Since by \eqref{gen} we have $B_{1}(x)=\sum_{j=0}^K \beta_{j-1}\log(1/A_1^j(x))+\beta_{K}(x)B_{1}(A_{1}^{K+1}(x))$   we get
\begin{equation}\label{DK}
    B_{1}(x)\geq \varphi_{K}(x) \quad \text{ for all } K\in \N, \ \ \ \mbox{ where } \varphi_{K}(x):= \sum_{j=0}^K \beta_{j-1}\log(1/A_1^j(x))+\beta_{K}(x)B_{1}(g).
\end{equation}

Note that $\varphi_{K}(g^2)=B(g^2)$; moreover $\varphi_{K}$ is smooth and convex on the $K$-th cylinder containing $g^2$ (recall that convexity follows from equation \eqref{generalterm}  ).
We now examine what equation \eqref{DK} gives for\footnote{Note that we omit the case $K=2$ since it turns out to be irrelevant for our purpose.} $K=0,1,3$; on the cylinder containing $g^2$ the building blocks of $\varphi_K$ will be $\omega_j(x):=\beta_{j-1}(x)\log(1/A_1^j(x))$ with $0\leq j \leq 3$ where 
\begin{eqnarray*}
    \omega_0(x)= -\log x \ \ x\in(0,1) & \ & \omega_1(x)= x [\log x-\log(1-2x)] \ \ x\in (\frac{1}{3},\frac{1}{2}) \\
  \omega_2(x)= \left(2  x - 1\right) \left[\log\left(3  x - 1\right) - \log\left(1-2  x \right)\right]  \ \ x\in (\frac{1}{3},\frac{2}{5})
& \ & \omega_3(x)= \left(3 x - 1\right) \left[\log\left(3x - 1\right) - \log\left(2-5 x \right)\right]  x\in (\frac{3}{8},\frac{2}{5})
\end{eqnarray*}

\begin{figure}[h]
    \centering
    \includegraphics[width=0.8\textwidth]{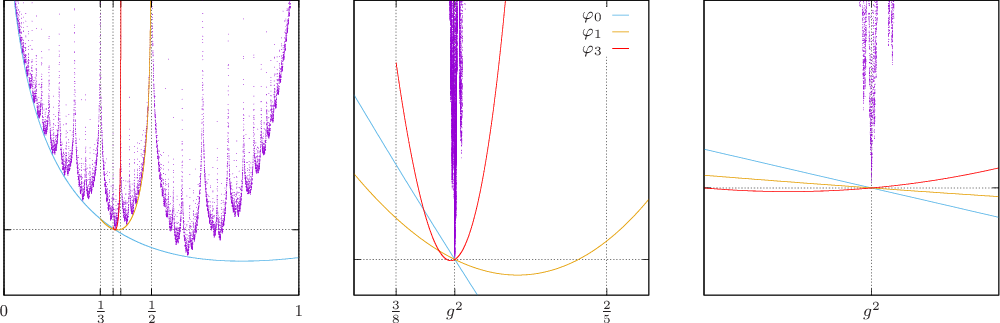}
    \caption{The three curves that bound $B_1$ from below, seen at three different level of zoom around the point $(g^2,B_1(g^2))$.}
    \label{fig:lowerbound@g2}
\end{figure} 

When $K=0$ (see Figure \ref{fig:lowerbound@g2} Left, azure line)  we have $
B_{1}(x)\geq \varphi_{0}(x)$ for $ x\in[0,1]$,

and   $\varphi_{0}(g^2)= B_{1}(g^2)$
    with $\varphi'_{0}(g^2)<0.$ This implies
\begin{equation*}
\varphi_{0}(x)\geq B_{1}(g^2) \text{ for all } x\in[0,g^2].
\end{equation*}
When $K=1$ (see Figure \ref{fig:lowerbound@g2} Center, orange line)  we have $B_{1}(x)\geq \varphi_{1}(x)$ where $x\in[1/3,1/2],$
and  $\varphi_{1}(2/5)>B_{1}(g^2),$ 
    with $\varphi'_{1}(\frac{2}{5})>0.$ Therefore, we have
\begin{equation*}
\varphi_{1}(x)\geq B_{1}(g^2) \text{ for all } x\in[2/5,1/2].
\end{equation*}

When $K=3$ (see Figure \ref{fig:lowerbound@g2} Right, red line)  we have $B_{1}(x)\geq \varphi_{3}(x)$ where $x\in[3/8,2/5]$ and $\varphi_{3}(g^2)= B_{1}(g^2)$
    with $\varphi'_{3}(g^2)>0.$
    This implies
\begin{equation*}
\varphi_{3}(x)\geq B_{1}(g^2) \text{ for all } x\in[g^2,2/5].
\end{equation*}

Since $[0,1/2]=[0,g^2]\cup [g^2,2/5]\cup[2/5,1/2].$ We get $B_{1}(x)\geq B_{1}(g^2)$ for all $x\in[0,1/2]$.

  \end{proof}
  
\section{Minima of $B_{\alpha}$ for $\alpha<1$}

In this section, we explore to what extent the result of \cite{BaMa_20} can be exported to $B_\alpha$ for $\alpha<1$. One remarkable feature of this problem is that the minimum of $B_\alpha$ is constant when the parameter ranges on some  parameter intervals, but sometimes it jumps abruptly. The first occurrence of this phenomenon can be observed in a left neighbourhood of $\alpha=g$, but it seems to be a pattern which occurs repeteadly in parameter space.

\subsection{Minima of $B_{\alpha}$ when $g<\alpha<1.$}

Here, we prove Theorem \ref{Bingo}:  for any real number $\alpha\in(g,1),$  the minimum of $B_\alpha$ is attained at $g.$ It all boils down to the following lemma:
\begin{lem}
For all $\alpha\in(g,1),$
\begin{equation}\label{Bing1}
B_{\alpha}(x)\geq B_{1}(x).
\end{equation}
\end{lem}
Indeed, if $\alpha\in(g,1)$ then $B_{\alpha}(g)=B_{1}(g).$ Using the fact that $g$ is the minimum of $B_{1},$  inequality \eqref{Bing1} implies
\begin{equation*}
B_{\alpha}(x)\geq B_{1}(x)\geq B_{1}(g) \text{ for all } x.
\end{equation*}

Therefore, to prove Theorem \ref{Bingo} we only need to prove inequality \eqref{Bing1}.

\begin{proof}

Recall that $\beta_{k}(x)=|xq_{k}-p_{k}|$ where $\frac{p_{k}}{q_{k}}$ is the $k^{th}$ $\alpha$-convergent of $x.$ By using \cite[Lemma 1.8.]{MaMoYo_97}, we can write
$$p_{k}=P_{n(k)}   \quad q_{k}=Q_{n(k)},$$ 
where $\frac{P_{n}}{Q_{n}}$ is the $n^{th}$ convergent of the regular continued fraction of $x.$ Moreover one has that $n(k)-n(k-1)\in \{ 1,2\}$ i.e. either $n(k-1)=n(k)-1$ or $n(k-1)=n(k)-2.$ Therefore setting $\tilde\beta_{n}(x)=|Q_{n}x-P_{n}|$ we get $\beta_{k}=\tilde\beta_{n(k)}.$

\begin{align}
B_{\alpha}(x)&= \sum^{\infty}_{k=0} \beta_{k-1}(x)\log(1/A_{\alpha}^{k}(x))\notag\\
&= \sum^{\infty}_{k=0} \beta_{k-1}(x)[\log\beta_{k-1}(x)-\log\beta_{k}(x)]\notag\\
&= \sum^{\infty}_{k=0} \tilde\beta_{n(k-1)}(x)[\log\tilde\beta_{n(k-1)}(x)-\log\tilde\beta_{n(k)}(x)].\label{Bing3}
\end{align}

Now 
\begin{align*} \tilde\beta_{n(k-1)}(x)[\log\tilde\beta_{n(k-1)}(x)-\log\tilde\beta_{n(k)}(x)]=
\begin{cases}\tilde\beta_{n(k)-1}(x)[\log\tilde\beta_{n(k)-1}(x)-\log\tilde\beta_{n(k)}(x)]
&\text{ if } n(k-1)=n(k)-1\\
\tilde\beta_{n(k)-2}(x)[\log\tilde\beta_{n(k)-2}(x)-\log\tilde\beta_{n(k)}(x)] &\text{ if } n(k-1)=n(k)-2.
\end{cases}
\end{align*}

Further,
\begin{multline*}
  \tilde\beta_{n(k)-2}(x)[\log\tilde\beta_{n(k)-2}(x)-\log\tilde\beta_{n(k)}(x)]
 \tilde\beta_{n(k)-2}(x)[\log\tilde\beta_{n(k)-2}(x)-\log\tilde\beta_{n(k)-1}+ \log\tilde\beta_{n(k)-1}-\log\tilde\beta_{n(k)}(x)]\\
\geq  \tilde\beta_{n(k)-2}(x)[\log\tilde\beta_{n(k)-2}(x)-\log\tilde\beta_{n(k)-1}]+\tilde\beta_{n(k)-1}(x)[\log\tilde\beta_{n(k)-1}-\log\tilde\beta_{n(k)}(x)],
\end{multline*}
 the second last inequality follows by the fact that $\tilde\beta_{n-2}\geq\tilde \beta_{n-1}.$

Therefore, \eqref{Bing3} implies 

\begin{align*}
B_{\alpha}(x)&= \sum^{\infty}_{k=0} \tilde\beta_{n(k-1)}(x)[\log\tilde\beta_{n(k-1)}(x)-\log\tilde\beta_{n(k)}(x)]\\
&\geq    
 \sum^{\infty}_{k=1}\sum_{j=n(k-1)}^{n(k)-1}
\tilde\beta_{j-1}(x)[\log\tilde\beta_{j-1}(x)-\log\tilde\beta_{j}(x)]\\
&=\sum_{j=0}^{\infty}\tilde\beta_{j-1}(x)[\log\tilde\beta_{j-1}(x)-\log\tilde\beta_{j}(x)]=B_{1} (x).
\end{align*}

\end{proof}

As a consequence of Theorem \ref{Bingo} we get the following
\begin{cor}\label{immediate}
Let $g<\alpha<1.$ If $\nu=S\cdot g$ is a noble number such that 
\begin{equation*}
A^{k}_{\alpha}(\nu)\neq \alpha \text{ for all } k \leq |S|. 
\end{equation*}
Then $\nu$ is a local cusp-like minimum for $B_\alpha$. For example $\nu=g^{2}$ is a local minimum for $B_\alpha$ for all $g<\alpha<1.$
\end{cor}
The proof of Corollary \ref{immediate} is an immediate consequence of Theorem \ref{Bingo} and  
the fact that $A^{k}_{\alpha}$ is smooth at $\nu$ for all $k\leq |S|$. 

\subsection{Local minima for $B_{1/2}.$} \label{Ca2}
 In order to prove analogous results for other values of $\alpha,$ it is reasonable to start from $\alpha=1/2.$ In this case the absolute minimum will occur at $\gamma=\sqrt{2}-1$ where $B_{1/2}$ has a cusp. To prove it, we can adapt the  same techniques used in the proofs of Theorem \ref{BM} and Theorem \ref{AIM}.

 Throughout this subsection, we will assume $\alpha=1/2$, and thus we will focus on the Brjuno function $B_{1/2}$ associated with the nearest integer continued fractions. 
 
We will show that the minimum of $B_{1/2}$ is attained at the silver number $\gamma,$ see Figure \ref{2025}.  Using the symmetry and the translation  invariance  of $B_{1/2}$ it is enough to prove the following:
 \begin{thm}\label{minima}
 $\min_{x\in (0,1/2)}B_{1/2}(x)=B_{1/2}(\gamma)$ where $\gamma=\sqrt 2 -1$.\end{thm}

 \begin{figure}[h]
 \centering
  \includegraphics[width=0.8\textwidth]{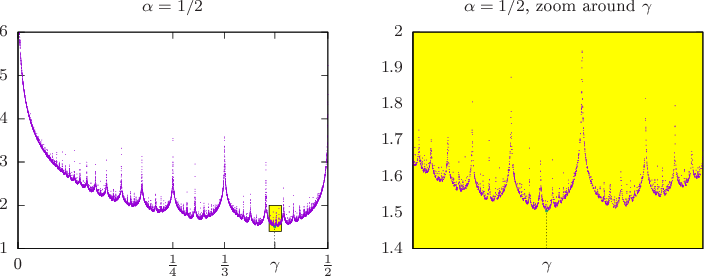}
 \caption{The graphs of $B_{1/2}$ with a zoom around the global minimum at $\gamma=\sqrt{2}-1.$ }\label{2025}
 \end{figure}  

The strategy of the  proof of Theorem \ref{minima} closely follows \cite{BaMa_20} and is based on four intermediate steps. 

Define $C:=\inf_{x\in[0,1/2]} B_{1/2}(x),$ the first step is a consequence of  Proposition \ref{LSCa} for $\alpha=1/2.$
\begin{step}
    $C=B_{1/2}(r)$ for some $r\in(0,1/2).$
\end{step}

\begin{step}\label{P3}
    Let $r\in(0,1/2).$ Then for all $K\in \N,$ we have 
\begin{equation*}
C=B_{1/2}(r)\geq \frac{B^{(K)}_{1/2}(r)}{1-\beta_{K}(r)}.
\end{equation*}
\end{step}

\begin{proof}
By using \eqref{gen} for $\alpha=1/2$ we have
\begin{align*}
C=B_{1/2}(r)=B^{(K)}_{1/2}(r)+\beta_{K}(r)B_{1/2}(A_{1/2}^{K+1}(r)) &\geq  B_{1/2}^{(K)}(r)+C\beta_{K}(r)=\frac{B^{(K)}_{1/2}(r)}{1-\beta_{K}(r)}.
\end{align*}
\end{proof}

\begin{step}\label{P4}
For all $K\in \N,$ we have 
\begin{equation*}
B_{1/2}(\gamma)= \frac{B^{K}_{1/2}(\gamma)}{1-\beta^{K}_{1/2}(\gamma)}.
\end{equation*}
\end{step}
\begin{proof}
Using \eqref{gen} now for $r=\gamma$ we obtain 
\begin{align*}
B_{1/2}(\gamma)=B^{K}_{1/2}(\gamma)+\beta_{K}(\gamma)B_{1/2}(A_{1/2}^{K+1}(\gamma)),
\end{align*}
since $A_{1/2}^{K+1}(\gamma)=A_{1/2}(\gamma)=\gamma$ for any $K\in\N$ therefore the required result follows.
\end{proof}

\begin{step}\label{P5}
Let $r\in(0,1/2)$ such that $C=B_{1/2}(r).$ Then $r\geq \gamma.$
\end{step}
\begin{proof}
From Step \ref{P3} and \ref{P4} with $K=0$ and by the definition of $C,$ we have
\begin{equation*}
\frac{B^{(0)}_{1/2}(\gamma)}{1-\beta_{0}(\gamma)}=B_{1/2}(\gamma)\geq C= B_{1/2}(r)\geq \frac{B^{(0)}_{1/2}(r)}{1-\beta_{0}(r)}.
\end{equation*}

Thus for $0<x< 1/2,$ 
\begin{align*}
\frac{B^{(0)}_{1/2}(x)}{1-\beta_{0}(x)}=\frac{\ln\frac{1}{x}}{1-x}.
\end{align*}
Let $h(x)=\frac{\ln\frac{1}{x}}{1-x}.$ Then 
$$h'(x)=\frac{-x^{-1}(1-x)+\ln (1/x)}{(1-x)^2}.$$ 
Since $h'(x)<0$ on $(0,1/2)$, 
the function $h$ is strictly decreasing on $(0,1/2),$ hence $\gamma\leq r.$

\end{proof}

\begin{pro}
Let $r\in(0,1/2)$ such that $C=B_{1/2}(r).$ Then $r=\gamma.$
\end{pro}
\begin{proof}
We need to show $\gamma>r.$ From Step \ref{P3} and Step \ref{P4} with $K=1$ and by the definition of $C,$ we have
\begin{align*}
\frac{B^{(1)}_{1/2}(\gamma)}{1-\beta_{1}(\gamma)}=B_{1/2}(\gamma)\geq C=B_{1/2}(r)\geq \frac{B^{(1)}_{1/2}(r)}{1-\beta_{1}(r).}
\end{align*}
Note that for  $2/5 <x<1/2,$ $A_{1/2}=\frac{1}{x}-2.$
Therefore, let 
\begin{align*}
f(x)&=\frac{\ln(1/x)}{1-xA_{1/2}}+\frac{x\ln(\frac{1}{A_{1/2}(x)})}{1-xA_{1/2}(x)}\\
&=\frac{\ln(1/x)}{2x}+\frac{1}{2}\ln\frac{x}{1-2x}.
\end{align*}
Then 
\begin{align*}
f'(x)=\frac{(1-2x)\ln x+3x-1}{2x^2(1-2x)}.
\end{align*}
Since $(1-2x)>0$ for $x\leq 1/2$, the sign of $f'(x)$ depends only on $g(x):=(1-2x)\ln x+3x-1.$
It is easy to see that $g$ is strictly increasing on the interval $(0,1/2]$ and consequently $f$ is increasing on $(2/5,1/2]$ (because $g'(x)$ and $f'(x)$ are positive on (2/5,1/2)). Thus $\gamma\geq r$ and combining it with the estimate of Step \ref{P5}, we finally get $\gamma=r.$
\end{proof}

\subsubsection{Scaling properties of $B_{1/2}.$}

\begin{thm}\label{SLB1/2}

Let $\gamma=\sqrt{2}-1.$ There exists a constant $c>0$ such that
\begin{equation*}B_{1/2}(x)-B_{1/2}(\gamma)\geq c|x-\gamma|^{1/2},\end{equation*}
for all $x\in (0,\frac{1}{2}).$

\end{thm}
The proof of Theorem \ref{SLB1/2} follows almost exactly on the same line of investigations as for the case $\alpha=1,$ with some additional arguments specific to the settings of $\alpha=1/2,$ which we will outline for completeness.
\begin{proof}
Let $\Psi(x)=\frac{1}{2+x}$ and define the recursive relation 
\begin{equation*}
x_{n+1}=\Psi(x_{n})   \quad \text{ where } x_{0}\in (0,2/5).
\end{equation*}
 
From \eqref{BFE} we have 
\begin{equation*}
B_{1/2}(\Psi(x_{n}))=-\log(\Psi(x_{n}))+\Psi(x_{n})B(x_{n})
\end{equation*}
and 
\begin{equation*}
B_{1/2}(\Psi(\gamma))=-\log(\Psi(\gamma))+\Psi(\gamma)B(\gamma).
\end{equation*}
Note that $\Psi(\gamma)=\gamma.$

Now 
\begin{align*}
B_{1/2}(x_{n+1})-B_{1/2}(\gamma)=-\log\frac{x_{n+1}}{\gamma}+x_{n+1}[B_{1/2}(x_{n})-B_{1/2}(\gamma)]+B_{1/2}(\gamma)[x_{n+1}-\gamma].
\end{align*}

Setting $\mathcal E_{n}:=B_{1/2}(x_{n})-B_{1/2}(\gamma)$ and $\delta_{n}:=x_{n}-\gamma$ and observing that $-\log t\geq 1-t$ we get 
\begin{equation*}
\mathcal E_{n+1}\geq x_{n+1}\mathcal E_{n}-l \delta_{n+1},
\end{equation*}
where $l:=\frac{1}{\gamma}-B_{1/2}(\gamma)$ and $l> 0.$
Hence
\begin{align*}
\mathcal E_{n+2}\geq x_{n+2}x_{n+1}\mathcal E_{n}-l (x_{n+2}\delta_{n+1}+\delta_{n+2}).
\end{align*}
Note that 
\begin{align*}
x_{n+2}\delta_{n+1}+\delta_{n+2}&=x_{n+2}(1+x_{n+1})-\gamma(1+x_{n+2}),\\
&= x_{n+2}(2+x_{n+1})-\gamma(2+x_{n+2})+\gamma-x_{n+2}=1-\gamma(2+x_{n+2})+\gamma-x_{n+2}.
\end{align*}
Both these terms $1-\gamma(2+x_{n+2}),$ $\gamma-x_{n+2}$  are negative if and only if $x_{n+2}>\gamma$ i.e. if $n$ is odd.

Therefore for all $n\ge 1,$ 
\begin{align*}
\mathcal E_{2n+1}\geq x_{2n+1}x_{2n} \mathcal E_{2n-1}.
\end{align*}
Repeating the similar arguments as used for $B_{1},$ we have 
\begin{equation} \label{EstB}
\mathcal E_{2n+1}\geq c_{3} \gamma^{2n} \text{ for some constant } c_{3} .
\end{equation}
This estimate shows that $B_{1/2}$ has a cusp like minimum at $\gamma,$ namely 
\begin{equation*}
B_{1/2}(x)-B_{1/2}(\gamma)\geq c|x-\gamma|^{\tau}        \quad \text{ with }
\end{equation*}
 $\tau=\frac{1}{2}$
 where the value of $\tau$ is obtained by using the fact that 
$\delta_{2n+1}\asymp \gamma^{4n}$ for any $n\geq1$ and comparing it with the estimates \eqref{EstB}.
\end{proof}

\subsection{Minima of $B_\alpha$ for $\alpha$ in a neighbourhood of $1/2$}

 In this section we will prove Theorem \ref{monotonic}; the core of the proof is contained in the following lemma. 
\begin{lem}\label{l:matching}
    Let $\phi(t):= \frac{1+t}{2+t}$ and let $\alpha, \alpha' \in (\gamma, g),$ such that one of the following is true
    $$(R) \ \ \frac{1}{2} \leq \alpha'<\alpha\leq \phi(\alpha')\leq g, \ \ \ \ \ \ 
    (L) \ \ \gamma < \alpha<\alpha'\leq\frac{1}{2}.$$
     Let $x\in \R$, let $x_0\in [0,\overline{\alpha}]$ be the representative of $x$ in $[0,\overline{\alpha}]$, and let  $x_k=A_\alpha^k(x_0)$ denote the orbit of $x_0$. Then we have the following relation:
    \begin{equation}\label{differenceB}
B_\alpha(x)- B_{\alpha'}(x)=\sum_{k} \beta_{k-1}h_I(x_k), 
\end{equation}
with $\beta_{k-1}=x_0x_1...x_{k-1}$, $h_I(y)=h(y)\chi_I(y)$ where $\chi_I$ is the charateristic function of the interval 
$$I=\begin{cases}
    [\alpha',\alpha) & \mbox{in case } (R)\\
(1-\alpha', 1-\alpha] & \mbox{ in case } (L)
\end{cases}$$ and
$$ h(y)=-\log y + \log(1-y) +y \log y- (1-y)\log(1-y)-(2y-1)\log(2y-1).\ 
$$

\end{lem}

\begin{proof}

Let $\alpha, \alpha'$ be such that (R) holds and let  $x$ be a real number. For now, suppose that $x\in [0, \alpha')$. In order to compare $B_\alpha(x)$ and $B_{\alpha'}(x),$ we need to analyze how the two orbits $x_k:=A_\alpha^k(x)$ and  $x'_k:=A_{\alpha'}^k(x)$ behave. For $k=0,$ the two orbits coincide and they will continue together until the first index $k_0$ for which $x_k$ enters the interval $[\alpha',\alpha)$, in that case $x_{k_0}'=1-x_{k_0} \in (1-\alpha, 1-\alpha'] \subset (\frac{1}{\alpha+2}, \frac{1}{2})$ (to establish the last inclusion we use the hypothesis $\alpha\leq \phi(\alpha')$). At the following step, we will have
$$x_{k_0+1}=2-\frac{1}{x_{k_0}}=\frac{2 x_{k_0}-1}{x_{k_0}}, \ \ \ \ \ \ x'_{k_0+1}=\frac{1}{x'_{k_0}}-2=\frac{1}{1-x_{k_0}}-2=\frac{2x_{k_0}-1}{1-x_{k_0}}.$$
Let us note that 
$$\frac{1}{x_{k_0+1}}-\frac{1}{x'_{k_0+1}}=1,$$
and this implies that at the next step,  either both $x'_{k_0+2}$ and $x_{k_0+2}$ belong to $[0,\alpha')$ and $x'_{k_0+2}=x_{k_0+2},$ or $x_{k_0+2}\in [\alpha',\alpha)$ and  $x'_{k_0+2}=1-x_{k_0+2}$. Therefore, the pair of points $(x_k,x_k')$ at each step is in one of the following three states:
$$(A) \ x'_k=x_k,\ \ \ \ (B)\  x'_k=1-x_k, \ \ \ \ (C) \ \frac{1}{x_{k}}-\frac{1}{x'_{k}}=1,$$
and all the possible transitions are given by the graph in Figure \ref{ABC}:

\begin{figure}[h]
 \centering
 \includegraphics[width=0.3\textwidth]{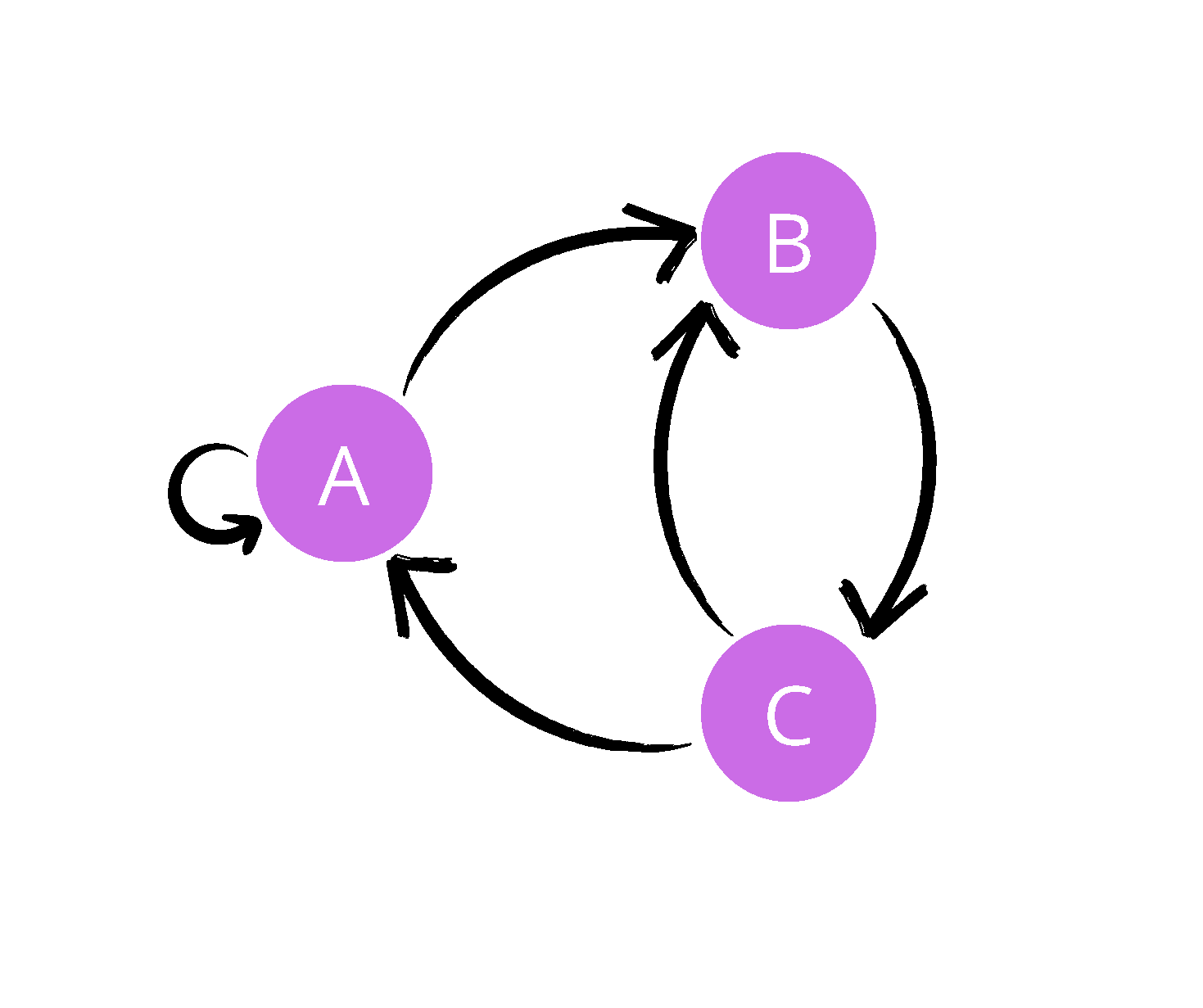}
 \caption{Transition diagram}  \label{ABC} \end{figure} 

 More generally, if the value $x \in \R$ (and not necessarily in $[0, \alpha')$ as above), one can find two values $x_0\in [0, \alpha]$ and $ x'_0 \in [0, \alpha']$ 
such that $$B_\alpha(x)=B_\alpha(x_0), \ \ \ B_{\alpha'}(x)=B_{\alpha'}(x'_0), \ \ \ \ \mbox{and} \ \ x'_0=x_0\in [0,\alpha') \ \mbox{ or } x_0 \in [\alpha',\alpha),  x'_0=1-x_0.$$
This means that the above analysis still holds for general $x\in \R,$ the only difference being that the starting state might be (B) rather than  (A).
In order to compare the values of two different $\alpha$-Brjuno functions at the same point, let us call $k_0,$ the minimum index such that the pair $(x_{k_0},x'_{k_0})$ is in state (B).  Using the functional equation, we then obtain:
\begin{equation}\label{syncro}
    \begin{array}{ll}
  B_\alpha(x)= B_\alpha(x_{0})= &B_\alpha^{k_0-1}(x_{0})+\beta_{k_0-1}B_\alpha(x_{k_0})  \\
    B_{\alpha'}(x)= B_{\alpha'}(x'_{0})= &B_{\alpha'}^{k_0-1}(x'_{0})+\beta'_{k_0-1}B_{\alpha'}(x'_{k_0}).
\end{array}
\end{equation} 

Note that if $k_0=0,$ both  equations above are trivial identities. Otherwise, by the definition of $k_0,$ it follows  that $B_\alpha^{k_0-1}(x_{0})=B_{\alpha'}^{k_0-1}(x'_{0})$ and $\beta_{k_0-1}=x_0...x_{k_0-1}=x'_0...x'_{k_0-1}=\beta'_{k_0-1}$.
On the other hand, setting $y:=x_{k_0} \in [\alpha',\alpha)$ (so that $x'_{k_0}=1-y$) we can use the functional equation a couple  more times to obtain 
$$ \begin{array}{rcl}
  B_\alpha(x_{k_0})= & \log \frac{1}{y}+yB_\alpha(2-\frac{1}{y}) =
  & \log \frac{1}{y}+y[ \log\frac{y}{2y-1}+\frac{2y-1}{y}B_\alpha(x_{k_0+2})]  \\
    B_{\alpha'}(x'_{k_0})= B_{\alpha'}(1-y)= & \log \frac{1}{1-y}+(1-y)B_{\alpha'}(\frac{1}{1-y}-2)=& \log \frac{1}{1-y}+(1-y)[\log \frac{1-y}{2y-1}+ \frac{2y-1}{1-y}B_{\alpha'}(x'_{k_0+2}) ].
\end{array}
$$
So, subtracting the equations \eqref{syncro}, and using these two formulas, we get
$$ B_\alpha(x)- B_{\alpha'}(x)=\beta_{k_0-1} h(x_{k_0})+\beta[B_\alpha(x_{k_0+2})-B_{\alpha'}(x'_{k_0+2}) ]. $$
where $\beta=\beta_{k_0+1}=\beta'_{k_0+1}$ and
\begin{align*}h(y)&=\log \frac{1}{y}+y \log\frac{y}{2y-1}-\log \frac{1}{1-y}-(1-y)\log \frac{1-y}{2y-1}\\
&=-\log y + \log(1-y) +y \log y- (1-y)\log(1-y)-(2y-1)\log(2y-1).
\end{align*}
Note that the pair $(x_{k_0+2},x'_{k_0+2})$ is either in state $(A)$ or in state $(B,)$ so we can iterate the same argument and get 
\begin{equation*}
B_\alpha(x)- B_{\alpha'}(x)=\sum_{k\in J_B} \beta_{k-1}h(x_k) \ \ \ \ \mbox{ where } 
J_B=\{k \in \N \ : \ (x_k,x'_k) \ \mbox{ is in state $(B)$}\}.
\end{equation*}
This is the same as formula \eqref{differenceB} since
$$k \in J_B \iff x_k\in [\alpha',\alpha),$$
and the claim for case $(R)$ is proved.

The very same argument works in case $(L),$ namely for $\gamma <\alpha < \alpha'\leq 1/2 $.

 \end{proof}
\begin{figure}[h]
 \centering
\includegraphics[width=0.8\textwidth]{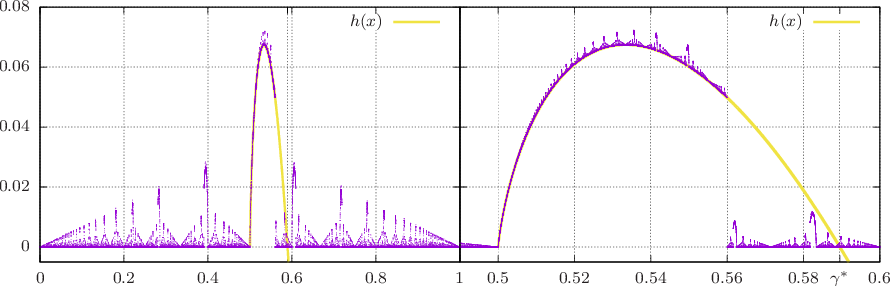}
 \caption{The graph of $B_{\alpha}-B_{\alpha'}$ for $\alpha'=0.5$ and $\alpha=0.56$ with a zoom around $(\alpha',\alpha)=[0.5, 0.56].$ }\label{nimpa}
 \end{figure}  
 \begin{rem}
  Even if we never mentioned it, the proof of Lemma  \ref{l:matching} relies on the property of matching, which is a peculiar feature of $\alpha$-continued fractions. We will give a quick overview of this property in the last subsection, but for a complete account we refer to \cite{CT} or also \cite{BaCaLe_24} for a viewpoint closer to our focus.    
 \end{rem}

Note that the function $h$ is continuous  on $[1/2,g]$, hence it is bounded. Moreover  $h$ is strictly positive for $x\in ( 1/2,\gamma^*),$ where $\gamma^*\approx 0.5895...\in (1-\gamma, 3/5)$. Therefore, if $1/2\leq \alpha' \leq\alpha<\gamma^*$, then $B_\alpha(x)\geq B_{\alpha'}(x).$ Since by ergodicity of $A_\alpha,$ the orbit of almost every $x\in [0,\alpha)$ visits the open interval $(\alpha',\alpha),$ we get that this difference is actually strictly positive for almost every $x$. 
The same argument works also in the $(L)$ case, giving that for all $\alpha, \alpha'$ such that
$\gamma < \alpha<\alpha'\leq\frac{1}{2},$ the inequality 
$B_\alpha(x) > B_{\alpha'}(x)$ holds for almost every $x,$ see Figure \ref{nimpa} .

The above considerations together with Theorem \ref{monotonic} and
Theorem \ref{minima}  immediately imply the following: 
\begin{cor}
    The minimum point and value do not change as $\alpha$ ranges over $(\gamma, \gamma^*]$ i.e.,

    $$\min_{x}B_{\alpha}(x)=B_{\alpha}(\gamma) \ \ \ \ \forall \ \alpha\in (\gamma, \gamma^*). $$
\end{cor}

Note that our results leave a gap in parameter space: in fact on the interval $(\gamma^*, g),$ we only have some numerical evidence that suggests that both the minimum point and value do not stay constant.

\subsection{Numerical evidence and conjectures.}
There are quite a few issues which remain open: questions for which we do not yet have a rigorous answer, but for which we can formulate conjectures based on solid numerical evidence. We list some of them below.

\subsubsection{Is the minimum of  $B_\alpha$ always attained?}
We believe the answer is negative. In fact, numerical evidence    suggests that $\inf_x B_g(x)=\sigma_0:=-3\log g,$ and this value is not a minimum, (see Figure \ref{fig:9}, Left).
Indeed, letting $\phi(x)=\frac{1+x}{2+x},$ it is easy to check that $\phi^n(g^2)\nearrow g$ and $B_g(\phi^n(g^2))\to \sigma_0$ as $n\to\infty;$ thus $\sigma_0 \geq \inf_x B_g(x)$. On the other hand, the numerical evidence\footnote{This fact could also be proven analytically.} shows that $\inf_x B_g(x)=\inf_{x\in (1/2,g)} B_g(x).$ However, it is also true that $B_g(x) > B_g(\phi(x))$ for all $ x \in (1/2,g)$ (see Figure  \ref{fig:9}, Right). This last property is an obstruction to the existence of  a minimum in  $(1/2,g)$ and is quite delicate to prove.

\begin{figure}[h]
    \centering
    \includegraphics[width=0.8\textwidth]{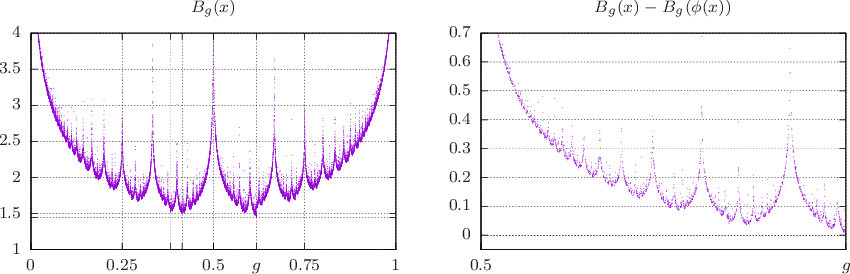}
    \caption{Left: the graph of $B_g$. Right: graph of $B_g(x)-B_g(\phi(x))$ on $(1/2,g)$.}
    \label{fig:9}
\end{figure}

However, we believe that if $\alpha \in (\gamma, g)$ the minimum is attained.

\subsubsection{What is the regularity of the function $\alpha \mapsto \inf_x B_\alpha(x)$ on $[1/2,1]$?}
From the numerical evidence, the function $\alpha \mapsto \inf_x B_\alpha(x)$ looks continuous\footnote{The function looks like a devil's staircase.} on $(\gamma,g),$ (even if the minimum point might jump see figure \ref{maximazoom}). However, this function seems to have a big jump at $\alpha=g$, and also at $\alpha=\gamma$. This latter discontinuity seems to be due to the fact that, in a neighbourhood of $\gamma,$ $\min B_\alpha(x)= B_\alpha(\gamma)$, and yet the value $B_\alpha(\gamma)$ changes abruptly when $\alpha$ goes below $\gamma$ because in this case $\gamma$ becomes preperiodic point (i.e. it is not a fixed point any more).
\begin{figure}[h]
  \centering
   \includegraphics[width=0.8\textwidth]{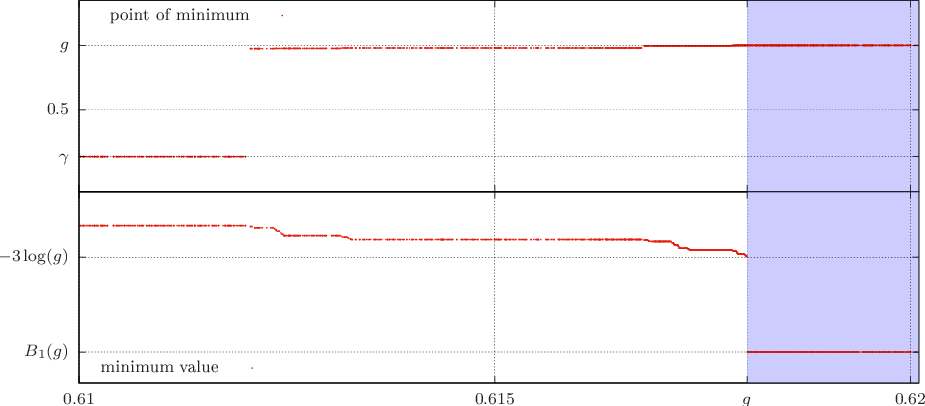}
  \caption{The graphs of the minimum point and the minimum value, zoomed on a range where they are not constant.} 
\label{maximazoom}  
   \end{figure} 

\subsubsection{What can we expect from $B_\alpha$ for  $\alpha< \gamma?$}
To answer this question, we have to say something more about the matching property within the family of $\alpha$-continued fractions $(A_\alpha)$. It is  known (see \cite{CT}) that parameter space splits into two sets: a ``stable set'' and a bifurcation set $\mathcal{E},$ where $\mathcal {E}$ is  defined below (see \eqref{EE}). 

The bifurcation set is a measure zero closed set which admits the following simple characterization in terms of the Gauss map: 
\begin{equation}\label{EE}
    \mathcal{E}=\{\alpha \in [0,1] \ : A_1^k(\alpha) \geq \alpha \ \ \forall \ k \geq 0\}. 
\end{equation}
The largest element of $\mathcal{E}$ is $g$, the second largest is $\gamma.$ However, $\mathcal{E}$ is not made up of isolated points. Indeed by the above  characterization, it is immediate to see that every irrational $\alpha=[0;a_1,a_2,a_3,...],$ with $a_1>a_j \ \forall j >1$ belongs to $\mathcal{E}$.

The complement of $\mathcal{E}$ is a countable union of its connected components which are open intervals. These open intervals are referred to as ``matching'' or ``synchronization''intervals,  for a reason that we explain below. Let $\alpha < \alpha'$ belong to the same matching interval and let them be not too far from each other, let us consider $x\in \R$ and call $x_0$ the canonical representative of $x$  in  $I_\alpha$,  $x'_0$ the canonical representative in $I_{\alpha'}.$ Consider the two orbits  $x_k=A_\alpha^k(x)$ and $x'_k=A_{\alpha'}^k(x)$; either the two orbits coincide forever (i.e. they are in state (A), with the notation used in Lemma \ref{l:matching}) or they eventually part for a certain index $k>0$, which can only happen if $x'_{k_0}=1-x_{k_0}$  (i.e. they enter in state (B)); note that this latter case might occur even for $k_0=0$. 

The effect of the matching property is that after a fixed number of $N$ iteration of $A_\alpha$ and $N'$ iterations of $A_{\alpha'}$ (where the  integers $N,N'$ depend only on the matching interval) the pair $(x_{k_0+N}, x'_{k_0+N'}),$ is again either in state (A) or (B), and $\beta_{k_0+N}=\beta_{k_0+N'}$. 

Note that, in general $N$ and $N',$ may be different. Indeed, if $\alpha, \alpha' \in (g,1],$ when the two orbits first part they satisfy $x'_k= 1-x_k$, but then $x_{k+1}=1/x_k-1.$ Hence $1/x'_k - 1/x_{k+1}=1$. This implies either $x'_{k+1}=x_{k+2}$ or $x'_{k+1}=1-x_{k+2}$,
i.e. $N'=1$, $N=2.$ Note that since $x'_k=x_kx_{k+1}$ it follows that $\beta'_k=\beta_{k+1}$.

Using this synchronization property, one can prove on every matching interval an analogue of formula \eqref{differenceB} obtaining that $B_\alpha(x)- B_{\alpha'}(x)=\sum_{k} \beta_{k-1}h(x_k) \chi_I(x_k)$, where $\chi_I$ is the characteristic function of $I_\alpha \setminus I{\alpha'}$ and  $h$ is  an analytic function depending on the interval.  We believe that this feature can be useful to exlpain why the graph of the minimum function  looks everywhere constant (see Figure \ref{figureminima} and Figure \ref {maximazoom}), and probably it is also useful to study the  continuity. On the basis of the numerical evidence and the above discussion it is natural to ask the following:
\begin{que}
  Let $\alpha \in [0,1]\setminus \mathcal{E}$.
Is $\inf_xB_\alpha (x)$ actually a minimum?
Is the function $\alpha \mapsto \inf_xB_\alpha (x)$ continuous at $\alpha_0$?  
\end{que}

However, it is worth noting that when $\alpha<\gamma$, characterizing the minimum point of $B_\alpha$ might be nontrivial, even in the cases when the numerical evidence is quite clear. For instance, when $\alpha=2/5,$ the numerical evidence suggests that $\min_x B_\alpha(x)= B_\alpha(\gamma)$, but in this case $\gamma$ {\bf is not} a fixed point of $A_\alpha$ (indeed $\gamma$ is only a eventually periodic point of $A_\alpha$). Therefore, one cannot just adapt the method which worked for $\alpha=1,1/2$.
The same issue arises for $\alpha=1/(n+1)$ with $n>1$; in this case, numerical evidence suggests that 
 the minimum of $B_\alpha$ is attained at the point 
$r=[0;1,n-1,\overline{1,n}],$ which again is not a fixed point of $A_\alpha$.

\noindent\textbf{Acknowledgements:}

The research of the first and third authors is supported by the research project ‘Dynamics and Information Research
Institute - Quantum Information, Quantum Technologies’ within the agreement between UniCredit
Bank and Scuola Normale Superiore.
The first author also acknowledges the support of the Centro di Ricerca Matematica Ennio de
Giorgi, SNS for providing excellent working conditions and research travel funds. 
The second author is partially supported by the PRIN Grant 2022NTKXC of the Ministry of University and Research (MUR), Italy. The second author acknowledges the support of the MIUR Excellence Department Project awarded to the Department of Mathematics, University of Pisa, CUP I57G22000700001.
The authors would like to thank the anonymous referee for the helpful comments, which have significantly improved the presentation of the paper.

\Addresses

\end{document}